\documentclass[11pt]{amsart}
\usepackage{amsbsy}
\usepackage{graphicx,epsfig,subfigure,psfrag}
\usepackage[utf8]{inputenc}
\begin{document}

\newtheorem{theorem}{Theorem}
\newtheorem{proposition}{Proposition}
\newtheorem{lemma}{Lemma}
\newtheorem{corollary}{Corollary}
\newtheorem{definition}{Definition}
\newtheorem{remark}{Remark}
\newtheorem{remarks}{Remarks}
\numberwithin{equation}{section} \numberwithin{theorem}{section}
\numberwithin{proposition}{section} \numberwithin{lemma}{section}
\numberwithin{corollary}{section}
\numberwithin{definition}{section} \numberwithin{remark}{section}

\title{\textbf{Existence of positive solutions for a Brezis--Nirenberg type problem involving an inverse operator}}

\author[P. \'Alvarez-Caudevilla]{P. \'Alvarez-Caudevilla}
\address{Departamento de Matem\'aticas, Universidad Carlos III de Madrid,
Av. Universidad 30, 28911 Legan\'es (Madrid), Spain}
\email{pacaudev@math.uc3m.es}

\author[E. Colorado]{E. Colorado}
\address{Departamento de Matem\'aticas, Universidad Carlos III de Madrid,
Av. Universidad 30, 28911 Legan\'es (Madrid), Spain}
\email{ecolorad@math.uc3m.es}

\author[A. Ortega]{A. Ortega}
\address{Departamento de Matem\'aticas, Universidad Carlos III de Madrid,
Av. Universidad 30, 28911 Legan\'es (Madrid), Spain}
\email{alortega@math.uc3m.es}

\thanks{All authors have been partially supported by the Ministry of Economy and Competitiveness of
Spain and FEDER under research project MTM2016-80618-P}
\thanks{The first author was also partially supported by the Ministry of Economy and Competitiveness of
Spain under research projects RYC-2014-15284.}

\date{\today}


\begin{abstract}
This paper is devoted to the existence of positive solutions for a problem related to a fourth-order differential
equation involving a nonlinear term depending on a second order differential operator,
$$(-\Delta)^2 u=\lambda u+ (-\Delta)|u|^{p-1}u,$$
in a bounded domain $\Omega\subset\mathbb{R}^N$, $N\geq 7$, and assuming homogeneous
Navier boundary conditions. In particular, we study a second order equation involving a nonlocal term of the form,
$$-\Delta u=\lambda (-\Delta)^{-1} u+|u|^{p-1}u,$$
under Dirichlet boundary conditions and we prove the existence of
positive solutions depending on the positive real parameter $\lambda>0$, up to the
critical value of the exponent $p$, i.e., when $1<p\leq 2^*-1$, where $2^*=\frac{2N}{N-2}$
is the critical Sobolev exponent. For $p=2^*-1$, this equivalence
leads us to a Brezis--Nirenberg type problem, cf. \cite{BN}, but,
in our particular case, the linear term is a nonlocal term. The effect that
this nonlocal term has on the equation changes the dimensions for which the classical
technique based on the minimizers of the Sobolev constant ensures the existence of
solution, going from dimensions $N\geq 4$ in the classical Brezis-Nirenberg problem,
to dimensions $N\geq7$ for this nonlocal problem.
\end{abstract}
\maketitle


\section{Introduction}\label{sec:intro}
\noindent In this work, we analyze the existence of positive solutions of a problem derived from the following fourth-order
equation under homogeneous Navier boundary conditions,
\begin{equation}\label{eq:CH}
\left\{\begin{array}{rll}
(-\Delta)^2 u  &=  \gamma u+ (-\Delta)|u|^{p-1}u&\quad\mbox{in}\quad \Omega\subset \mathbb{R}^{N},\\
u&=0  & \quad\mbox{on}\quad \partial\Omega,\\
-\Delta u&=0 & \quad\mbox{on}\quad \partial\Omega,
\end{array}\right.
\tag{$P^2_\gamma$}
\end{equation}
where $\gamma$ is a positive real parameter and $\Omega$ is a smooth bounded domain of $\mathbb{R}$, with $N\geq 7$. This important fact on the dimension will be under review along this work. In particular, positive solutions of \eqref{eq:CH} can be seen as positive steady-state solutions
 of the {\em fourth-order parabolic Cahn--Hilliard type equation},
\begin{equation*}
\frac{\partial u}{\partial t}+(-\Delta)^2 u= \gamma u +(-\Delta)|u|^{p-1}u,\quad\mbox{in}\quad \Omega\times\mathbb{R}_+,
\end{equation*}
assuming bounded smooth initial data $u(x,0)=u_0(x)$. The latter equation has been previously studied
in \cite{PV,AEGnegII} for bounded domains or the whole $\mathbb{R}^N$ but considering exponents $p$ in the subcritical range $1<p<2^*-1$, where $2^*=\frac{2N}{N-2}$ is the critical exponent of the embedding $H_0^1(\Omega) \hookrightarrow L^{p+1}(\Omega)$. In this work we extend the former range and we consider exponents $1<p\leq 2^*-1$, covering the critical exponent case. Let us recall that, because of the Sobolev Embedding Theorem, we have the compact embedding
\begin{equation}
\label{compact_emb}
H_0^1(\Omega) \subset\subset L^{p+1}(\Omega),
\end{equation}
for $2\leq p+1<2^*$, being a continuous embedding up to the critical exponent $p=2^*-1$. Moreover, given $u\in H_0^1(\Omega)$, because of the Sobolev inequality, there exist a positive constant $C=C(N,p)$ such that
\begin{equation}\label{sobolev}
\|u\|_{L^{p+1}(\Omega)}\leq C \|u\|_{H_0^1(\Omega)},
\end{equation}
for $2\leq p+1\leq 2^*$. Note that here, for the fourth-order elliptic problem \eqref{eq:CH}, the Sobolev's critical exponent we
are using is $2^* =\frac {2N}{N-2}$, because this operator has the representation,
\begin{equation*}
(-\Delta)^2u-(-\Delta)|u|^{p-1}u=(-\Delta)((-\Delta) u-|u|^{p-1}u),
\end{equation*}
so that, the necessary embedding features are governed by a standard second-order equation,
\begin{equation*}
-\Delta u=|u|^{p-1}u.
\end{equation*}
This is different from the usual critical problems with a bi-Laplacian operator of the form,
\begin{equation*}
(-\Delta)^2 u=\gamma u+ |u|^{p-1}u,
\end{equation*}
analyzed by Gazzola--Grunau--Sweers \cite{GGS}, where the Sobolev's critical exponent is $p_{S}=\frac {2N}{N-4}$. \newline
On the other hand, we also observe that \eqref{eq:CH} is not a variational problem. Nonetheless, applying $(-\Delta)^{-1}$
to the equation of \eqref{eq:CH}, we obtain the following non-local elliptic Dirichlet problem,
\begin{equation} \label{eq:nonlocal}
\left\{\begin{array}{ll}
-\Delta u=\gamma (-\Delta)^{-1} u + |u|^{p-1}u & \mbox{in} \,\,\, \Omega,\\
\quad\ \ u=0                      & \mbox{on} \,\,\, \partial\Omega,
\end{array}\right.
\tag{$P_{\gamma}$}
\end{equation}
which is a variational problem with the following associated Euler-Lagrange functional,
\begin{equation}
\label{fn:nonlocal}
    \mathcal{F}_{\gamma}(u)=\frac{1}{2}\int_{\Omega}|\nabla u|^2dx-\frac{\gamma}{2} \int_{\Omega} u(-\Delta)^{-1}u\, dx
        -\frac{1}{p+1} \int_{\Omega} |u|^{p+1}dx,
\end{equation}
 so that solutions of \eqref{eq:nonlocal} can be obtained as critical points of the Fr\'echet-differentiable
functional $\mathcal{F}_{\gamma}$ defined by \eqref{fn:nonlocal}. Here, as customary $(-\Delta)^{-1}u=v$, if
\begin{equation*}
 -\Delta v=u \,\,\, \mbox{in} \,\,\, \Omega, \quad v=0 \,\,\, \mbox{on}
 \,\,\, \partial \Omega.
\end{equation*}
Note that $(-\Delta)^{-1}$ is a positive linear integral compact operator from $L^2(\Omega)$ into itself, which is
well defined thanks to the Spectral Theorem. Next, we recall the following well-known facts about \textit{polyharmonic} operators of order $2m$ ($m\geq1$ an integer number) in smooth domains $\Omega$. The Navier boundary conditions for the operator $(-\Delta)^m$ are defined as
\begin{equation*}
u=\Delta u=\Delta^2 u=\ldots=\Delta^{k-1} u=0,\quad\mbox{on }\partial\Omega.
\end{equation*}
Clearly, the operator $(-\Delta)^m$ is the $m$-th power of the classical Dirichlet Laplacian in the sense of the spectral theory and it can be defined as the operator whose action on a function $u$ is given by
\begin{equation*}
\langle (-\Delta)^m u, u \rangle=\sum_{j\ge 1} \lambda_j^m|\langle u_1,\varphi_j\rangle|^2,
\end{equation*}
where $(\varphi_i,\lambda_i)$ are the eigenfunctions and eigenvalues of the Laplace operator $(-\Delta)$ with homogeneous
Dirichlet boundary data. Thus,
the operator $(-\Delta)^m$ is well defined in the space of functions that vanish on the boundary,
\begin{equation*}
H_0^m(\Omega)=\left\{u=\sum_{j=1}^{\infty} a_j\varphi_j\in L^2(\Omega):\ ||u||_{H_0^m(\Omega)}=\left(\sum_{j=1}^{\infty} a_j^2\lambda_j^m \right)^{\frac{1}{2}}<\infty\right\}.
\end{equation*}
Since the above definition allows us to integrate by parts, a natural definition of energy solution for problem \eqref{eq:nonlocal}
is given by critical points of the functional $\mathcal{F}_{\gamma}$ defined by \eqref{fn:nonlocal}. Moreover, we can rewrite the functional \eqref{fn:nonlocal} as,
\begin{equation*}
    \mathcal{F}_{\gamma}(u)=\frac{1}{2}\int_{\Omega}|\nabla u|^2dx-\frac{\gamma}{2}\int_{\Omega}|(-\Delta)^{-1/2}u|^2dx-\frac{1}{p+1}\int_{\Omega}|u|^{p+1}dx.
\end{equation*}

Additionally, we have a connection between problem \eqref{eq:CH} and a second order
elliptic system through problem \eqref{eq:nonlocal}. In particular, taking $\textstyle{w:=(-\Delta)^{-1}u}$,
problem \eqref{eq:nonlocal} provides us with the system,
\begin{equation}\label{s1}
    \left\{\begin{array}{l}
    -\Delta u = \gamma w+|u|^{p-1}u,\\
    -\Delta w = u,
    \end{array}\right.\quad \hbox{in}\quad \Omega,\quad (u,w)=(0,0)\quad\hbox{on}\quad \partial\Omega,
\end{equation}
which gives a different perspective to the problem in hand. In fact, we shall obtain the
main results of this paper following both perspectives with respect to the non-local
equation \eqref{eq:nonlocal} and the provided by considering a second order elliptic system. Moreover,
in order to obtain a variational system from problem \eqref{eq:nonlocal}, and since
$\gamma>0$, we take $\textstyle{v:=\sqrt{\gamma}w}$ in \eqref{s1} and we obtain the variational
system
\begin{equation}
\label{system}
\left\{\begin{array}{l}
-\Delta u=\sqrt{\gamma}v+|u|^{p-1}u,\\
-\Delta v=\sqrt{\gamma}u,
    \end{array}\right.\quad \hbox{in}\quad \Omega,\quad (u,v)=(0,0)\quad\hbox{in}\quad \partial\Omega,
\tag{$S_{\gamma}$}
 \end{equation}
whose associated Euler-Lagrange functional is
\begin{equation}
\label{fn:system}
    \mathcal{J}_{\gamma}(u,v)=\frac{1}{2} \int_{\Omega} |\nabla u|^2dx + \frac{1}{2} \int_{\Omega} |\nabla v|^2dx -\sqrt{\gamma}\int_{\Omega} uvdx -\frac{1}{p+1} \int_{\Omega} |u|^{p+1}dx.
\end{equation}
\begin{remark}
Because of the Maximum Principle, given $u$ a positive solution to \eqref{eq:nonlocal},
and setting $v=\sqrt{\gamma}(-\Delta)^{-1}u$, it follows that $v>0$ thus, the pair
$(u,v)=(u,\sqrt{\gamma}(-\Delta)^{-1}u)$ is a positive solution to \eqref{system} and vice versa,
given $(u,v)$ a positive solution to \eqref{system} it is immediate that $u(x)$ is a
positive solution to \eqref{eq:nonlocal}.
\end{remark}
Let us observe that, at the critical exponent $p=2^*-1$, problem \eqref{eq:nonlocal} can be seen as a
linear perturbation of the critical problem,
\begin{equation}\label{po70}
\left\{
\begin{tabular}{lcl}
$-\Delta u=|u|^{2^*-2}u$ & &in $\Omega\subset \mathbb{R}^{N}$, \\
\quad\ \ $u=0$ & &on $\partial\Omega$.
\end{tabular}
\right.
\end{equation}
for which, after applying the well-known result of Pohozaev, \cite{Poh}, one can prove
the non-existence of positive solutions under the star-shapeness assumption on the domain
$\Omega$. Moreover, the classical Brezis--Nirenberg problem,
\begin{equation}\label{BN83}
\left\{
\begin{tabular}{lcl}
$-\Delta u=\gamma u+ |u|^{2^*-2}u$ & &in $\Omega\subset \mathbb{R}^{N}$, \\
\quad\ \ $u=0$                                        & &on $\partial\Omega$.
\end{tabular}
\right.
\end{equation}
can be seen as well as a linear perturbation of problem \eqref{po70}. In his pioneering
paper, \cite{BN}, Brezis and Nirenberg proved that, for $N\geq4$, there exists a positive
solution to \eqref{BN83} if and only if the parameter $\gamma$ belongs to the interval $(0,\lambda_1)$,
being $\lambda_1$ the first eigenvalue for the Laplacian under homogeneous Dirichlet
boundary conditions. Note that, in our situation, the non-local term $\gamma(-\Delta)^{-1}u$ plays actually the role of $\gamma u$ in \eqref{BN83}. This important fact is under analysis in Section\;\ref{concomp}.

\vspace{0.4cm}

\underline{\bf Main results.} We prove the existence of positive solutions
of problem \eqref{eq:nonlocal} depending on the positive parameter $\gamma$. To do so, we will first show the
interval of the parameter $\gamma$ for which there is the possibility of having positive solutions. Next, applying the well-known Mountain Pass Theorem (MPT for short) \cite{AR}, we show that for the range
$2<p+1\leq 2^*$ there actually exists a positive solution to problem \eqref{eq:nonlocal}
provided
\begin{equation*}
0<\gamma<\lambda_1^*,
\end{equation*}
where $\lambda_1^*$ is the first eigenvalue of the operator $(-\Delta)^2$ under homogeneous Navier
boundary conditions, i.e. $\lambda_1^*=\lambda_1^2$ with $\lambda_1$ being the first
eigenvalue for the Laplacian under homogeneous Dirichlet boundary conditions.
If $2<p+1<2^*$ one might apply the MPT directly since, as we will show, our problem possesses
the mountain pass geometry and, thanks to the compact embedding \eqref{compact_emb},
the Palais--Smale condition is satisfied for the functional $\mathcal{F}_{\gamma}$ (see details
below in Section\;\ref{concomp}). On the other hand, at the critical exponent $2^*$, the compactness
of the Sobolev embedding is lost and check whether the Palais--Smale condition is satisfied becomes a delicate
issue to solve. To overcome this lack of compactness we apply a concentration-compactness argument based
on the Concentration-Compactness Principle due to P.-L. Lions, \cite{Lions}, which
allows us to prove the required Palais--Smale condition for $N\geq 7$. We prove the results for problem \eqref{eq:nonlocal} in Section \ref{concomp} and using similar ideas, for system \eqref{system} in Section\;\ref{concomp2}.\newline
Now we state the main results of this paper.
\begin{theorem}\label{Thequation_subcritical}
Assume $1<p<2^*-1$. Then, for every $\gamma\in(0,\lambda_1^*)$ there exists a positive solution
$u$ to problem \eqref{eq:nonlocal}.
\end{theorem}
\begin{theorem}
\label{Thequation}
Assume $p=2^*-1$. Then, for every $\gamma\in(0,\lambda_1^*)$, there exists a positive solution
$u$ to problem \eqref{eq:nonlocal} provided $N\geq7$.
\end{theorem}
Surprisingly, even though our problem \eqref{eq:nonlocal} is a non-local but also linear
perturbation of the problem \eqref{po70}, Theorem \ref{Thequation} addresses dimensions
$N\geq 7$, in contrast to the existence result of Brezis and Nirenberg about the linear
perturbation \eqref{BN83}, that covers the wider range $N\geq4$. In other words, the
non-local term $\gamma(-\Delta)^{-1}u$, despite of being just a linear perturbation, has an
important effect on the dimensions for which the classical Brezis--Nirenberg technique
based on the minimizers of the Sobolev constant still works.\newline
Finally, although the equivalence between the system \eqref{system} and the non-local
problem \eqref{eq:nonlocal} provides us with existence results for the system \eqref{system}
by means of Theorem \ref{Thequation_subcritical} and Theorem \ref{Thequation}, we prove independently the following.
\begin{theorem}\label{Thsystem_subcritical}
Assume $1<p<2^*-1$. Then, for every $\gamma\in(0,\lambda_1^*)$, there exists a positive solution
$(u,v)$ to system \eqref{system}.
\end{theorem}
\begin{theorem}
\label{Thsystem}
Assume $p=2^*-1$. Then, for every $\gamma\in(0,\lambda_1^*)$, there exists a positive solution
$(u,v)$ to system \eqref{system} provided $N\geq7$.
\end{theorem}
In the last section of the paper we extend our study to a high-order problem and we prove, under analogous hypotheses, that there exists a positive solution to the problem
\begin{equation}\label{general}
\left\{
\begin{tabular}{lcl}
$-\Delta u=\gamma (-\Delta)^{-m}u+ |u|^{p-1}u$ & &in $\Omega\subset \mathbb{R}^{N}$, \\
\quad\ \ $u=0$                                        & &on $\partial\Omega$.
\end{tabular}
\right.
\tag{$E_{\gamma,m}$}
\end{equation}
Due to the lack of a comparison principle for a higher order equations,
to obtain the existence results dealing with \eqref{general} we can not
tackle this problem directly, and we need to use a similar correspondence
to the one performed above for the problem \eqref{eq:CH}, now with an
elliptic system of $m+1$ equations.

\section{Existence of positive solutions for problem \eqref{eq:CH} via problem \eqref{eq:nonlocal}}\label{concomp}

\noindent In this section we carry out the proof of Theorem \ref{Thequation_subcritical}
and Theorem \ref{Thequation}. First, we establish a condition on the range of values of the
parameter $\gamma$ necessary for the existence of positive solutions to equation \eqref{eq:nonlocal}.
Let us consider the following generalized eigenvalue problem associated to \eqref{eq:nonlocal},
\begin{equation}\label{eiglin1}
\left\{
\begin{tabular}{lcl}
$-\Delta u=\lambda(-\Delta)^{-1}u$ & &in $\Omega\subset \mathbb{R}^{N}$, \\
\quad\ \ $u=0$&&on $\partial\Omega$.\\
\end{tabular}
\right.
\end{equation}
Then, we find that for the first eigenfunction $\varphi_1$ associated with the first
eigenvalue $\lambda_1^*$ in \eqref{eiglin1},
\begin{equation*}
\int_{\Omega} |\nabla \varphi_1|^2dx =\lambda_1^* \int_{\Omega} |(-\Delta)^{-1/2} \varphi_1|^2dx, \quad \hbox{with}\quad \varphi_1\in H^1_0(\Omega),
\end{equation*}
and, hence,
\begin{equation}\label{bieigen}
\lambda_1^*=\inf_{u\in H_0^1(\Omega)} \frac{\int_{\Omega} |\nabla u|^2dx}{ \int_{\Omega} |(-\Delta)^{-1/2} u|^2dx}.
\end{equation}
On the other hand, it is clear that substituting the first eigenfunction of the Laplace
operator under homogeneous Dirichlet boundary conditions, $\varphi_1$, into \eqref{eiglin1},
 it follows that $\lambda_1^*=\lambda_1^2$. Thus, by the very definition of the powers of the
Laplace operator, $\lambda_1^*$ coincides with the first eigenvalue of the operator $(-\Delta)^2$
under homogeneous Navier boundary conditions as well as the first eigenfunction of
\eqref{eiglin1} coincides with the first eigenfunction of the Laplace operator under
homogeneous Dirichlet boundary conditions. Now, we prove the following.
\begin{lemma}\label{cota}
Problem \eqref{eq:nonlocal} does not possess a positive solution when
$$\gamma \geq \lambda_1^*.$$
\end{lemma}
\begin{proof}
Assume that $u$ is a positive solution to \eqref{eq:nonlocal} and let $\varphi_1$ be a
positive first eigenfunction of the Laplacian operator in $\Omega$ under homogeneous Dirichlet
boundary conditions. Taking $\varphi_1$ as a test function for the equation of \eqref{eq:nonlocal}
we obtain,
\begin{align}\label{eq_gamma}
\int_{\Omega}\varphi_1(-\Delta)udx &=\gamma\int_{\Omega}\varphi_1(-\Delta)^{-1}udx+\int_{\Omega}|u|^{p-1}u\varphi_1dx\\
&>\gamma\int_{\Omega}\varphi_1(-\Delta)^{-1}udx\nonumber.
\end{align}
Thus, integrating by parts both sides of \eqref{eq_gamma},
\begin{equation*}
\lambda_1\int_{\Omega}u\varphi_1dx>\gamma\int_{\Omega}u(-\Delta)^{-1}\varphi_1dx=\frac{\gamma}{\lambda_1}\int_{\Omega}u\varphi_1dx.
\end{equation*}
Hence, $\gamma<\lambda_1^2=\lambda_1^*$.
\end{proof}
\begin{lemma}
\label{lezero}
The functional $\mathcal{F}_{\gamma}$ denoted by \eqref{fn:nonlocal} has the Mountain Pass geometry.
\end{lemma}
\begin{proof}
Without loss of generality we can take a function $g\in H_0^1(\Omega)$ such that
$\displaystyle\|g\|_{L^{p+1}(\Omega)}=1$. Then, taking a real number $t>0$ and applying
the Sobolev inequality \eqref{sobolev} together with \eqref{bieigen}, we find that,
\begin{align*}
\mathcal{F}_{\gamma}(tg)&= \frac{t^2}{2}\int_{\Omega}|\nabla g|^2dx-\frac{t^2\gamma}{2}\int_{\Omega}|(-\Delta)^{-\frac 1 2} g|^2dx-\frac{t^{p+1}}{p+1}\\
&\geq \frac{t^2}{2}\left(1-\frac{\gamma}{\lambda_1^*}\right)\int_{\Omega}|\nabla g|^2dx-\frac{t^{p+1}}{p+1}\\
&\geq \left(\frac{1}{2}\left(1-\frac{\gamma}{\lambda_1^*}\right)t^2-\frac{C}{(p+1)}t^{p+1}\right)\int_{\Omega}|\nabla g|^2dx\\
&>0
\end{align*}
for $t$ small enough, i.e.
\begin{equation*}
0<t^{p-1}<\frac{p+1}{2C}\left(1-\frac{\gamma}{\lambda_1^*}\right).
\end{equation*}
Thus, the functional $\mathcal{F}_{\gamma}$ has a local minimum at $u=0$, i.e.
\begin{equation*}
\mathcal{F}_{\gamma}(tg)>\mathcal{F}_{\gamma}(0)=0,
\end{equation*}
for any $g\in H_0^1(\Omega)$ provided $t>0$ is small enough. Also, it is clear that,
\begin{align*}
    \mathcal{F}_{\gamma}(tg)&=\frac{t^2}{2} \int_{\Omega} |\nabla g|^2dx - \frac{\gamma t^2}{2} \int_{\Omega} |(-\Delta)^{-1/2} g|^2dx -\frac{t^{p+1}}{p+1} \\
        &\leq \frac{t^2 }{2}\|g\|_{H_0^1(\Omega)}^2-\frac{t^{p+1}}{p+1}.
\end{align*}
Then,
\begin{equation*}
\mathcal{F}_{\gamma}(tg) \rightarrow -\infty,\quad \hbox{as}\quad t\to \infty,
\end{equation*}
and thus, there exists $\hat u \in H_0^1(\Omega)$ such that $\mathcal{F}_{\gamma}(\hat u)<0$.

\end{proof}
Now we turn our attention to the so-called Palais--Smale condition.
\begin{definition}
Let $V$ be a Banach space. We say that a sequence $\{u_n\}\subset V$ is a PS sequence
for a functional $\mathfrak{F}$ iff
\begin{equation}\label{convergencia}
\mathfrak{F}(u_n)\quad\hbox{is bounded and}\quad  \mathfrak{F}'(u_n) \to 0\quad\mbox{in}\ V'\quad \hbox{as}\quad n\to \infty,
\end{equation}
where $V'$ is the dual space of $V$. Moreover, we say that a PS sequence
$\{u_n\}\subset V$ satisfies a PS condition iff
\begin{equation}\label{conPS}
\{u_n\}\quad \mbox{has a convergent subsequence.}
\end{equation}
\end{definition}
In particular, given a PS sequence $\{u_n\}\subset V$ such that $\mathfrak{F}(u_n) \to c$,
if \eqref{conPS} is satisfied, we will say that the PS sequence satisfies a
PS condition at level $c$ for the functional $\mathfrak{F}$. Moreover, we say that the
functional $\mathfrak{F}$ satisfies the PS condition at level $c$ if every PS sequence
at level $c$ for $\mathfrak{F}$ possesses a convergent subsequence in $V$.\newline
For our problem, in the subcritical range the PS condition is always satisfied
at any level $c$ because of the compact Sobolev embedding. However, at the
critical exponent $2^*$ the problem is further complicated because of the lack
of compactness in the Sobolev embedding. We will overcome this issue applying
a concentration-compactness argument based on the Concentration-Compactness
Principle developed by P.-L. Lions, \cite{Lions}, proving that the functional
$\mathcal{F}_{\gamma}$ satisfies the PS condition for levels $c$ below a certain critical
value $c^*$ (to be determined).
\begin{lemma}\label{acotacion_eq}
Let $\{u_n\}$ be a PS sequence at level $c$ for the functional $\mathcal{F}_{\gamma}$, i.e.
\begin{equation*}
\mathcal{F}_{\gamma}(u_n) \rightarrow c,\quad \mathcal{F}_{\gamma}'(u_n) \rightarrow 0,\quad \hbox{as}\quad n\to \infty.
\end{equation*}
Then,
\begin{equation*}
\{u_n\}\quad \hbox{is bounded in}\quad  H_0^1(\Omega).
\end{equation*}
\end{lemma}
\begin{proof}
Since $\mathcal{F}_{\gamma}'(u_n) \rightarrow 0$ in $\left(H_0^1(\Omega)\right)'$, in particular we have
$\displaystyle\left\langle\mathcal{F}_{\gamma}'(u_n)|\frac{u_n}{\|u_n\|_{H_0^1(\Omega)}}\right\rangle\to 0$.
Thus, for any $\varepsilon>0$ there exists a subsequence, denoted again by $\{u_n\}$,
such that,
\begin{equation*}
\int_{\Omega} |\nabla u_n|^2dx-\gamma\int_{\Omega} |(-\Delta)^{-\frac{1}{2}}u_n|^2dx-\int_{\Omega} |u_n|^{p+1}dx=\|u_n\|_{H_0^1(\Omega)}\cdot o(1).
\end{equation*}
Moreover, since $\mathcal{F}_{\gamma}(u_n) \to c$,
\begin{equation*}
\frac{1}{2} \int_{\Omega} |\nabla u_n|^2dx-\frac{\gamma}{2}\int_{\Omega} |(-\Delta)^{-\frac{1}{2}}u_n|^2dx-\frac{1}{p+1} \int_{\Omega} |u_n|^{p+1}dx=c+o(1),
\end{equation*}
for $n$ big enough. Therefore, for a positive constant $\mu$ (to be determined below) we find that
\begin{equation*}
\mathcal{F}_{\gamma}(u_n)-\mu \left\langle \mathcal{F}_{\gamma}'(u_n)|\frac{u_n}{\|u_n\|_{H_0^1(\Omega)}}\right\rangle=c+\|u_n\|_{H_0^1(\Omega)}\cdot o(1).
\end{equation*}
That is,
\begin{align*}
&\left(\frac{1}{2}-\mu\right) \int_{\Omega} |\nabla u_n|^2dx-\left(\frac{1}{2}-\mu\right)\gamma\int_{\Omega} |(-\Delta)^{-\frac{1}{2}}u_n|^2dx
-\left(\frac{1}{p+1}-\mu\right) \int_{\Omega}|u_n|^{p+1}dx\\
&=c+\|u_n\|_{H_0^1(\Omega)}\cdot o(1).
\end{align*}
Hence, taking $\mu$ such that $\frac{1}{p+1}<\mu<\frac{1}{2}$,
\begin{equation*}
\left(\frac{1}{2}-\mu\right) \int_{\Omega} |\nabla u_n|^2dx-\left(\frac{1}{2}-\mu\right)\gamma\int_{\Omega} |(-\Delta)^{-\frac{1}{2}}u_n|^2dx
\leq c+\|u_n\|_{H_0^1(\Omega)}\cdot o(1),
\end{equation*}
and using \eqref{bieigen},
\begin{align*}
\left(\frac{1}{2}-\mu\right)\left(1-\frac{\gamma}{\lambda_1^*}\right)\int_{\Omega} |\nabla u_n|^2dx
&\leq \left(\frac{1}{2}-\mu\right) \int_{\Omega} |\nabla u_n|^2dx-\left(\frac{1}{2}-\mu\right)\gamma\int_{\Omega} |(-\Delta)^{-\frac{1}{2}}u_n|^2dx\\
&\leq c+\|u_n\|_{H_0^1(\Omega)}\cdot o(1).
\end{align*}
From here, we conclude
\begin{equation*}
\left(\frac{1}{2}-\mu\right)\left(1-\frac{\gamma}{\lambda_1^*}\right)\|u_n\|_{H_0^1(\Omega)}^2\leq c+\|u_n\|_{H_0^1(\Omega)}\cdot o(1).
\end{equation*}
Since $0<\gamma<\lambda_1^*$, it follows that
$\left(\frac{1}{2}-\mu\right)\left(1-\frac{\gamma}{\lambda_1^*}\right)>0$ and, thus, because
of the former inequality we conclude that the sequence $\{u_n\}$ is bounded in $H_0^1(\Omega)$.
\end{proof}

\begin{proof}[Proof of Theorem \ref{Thequation_subcritical}.]\hfill\break
Let us consider the subcritical case $1<p<2^*-1$. Given a PS sequence $\{u_n\}\subset H_0^1(\Omega)$
at level $c$, by Lemma \ref{acotacion_eq} and the Rellich-Kondrachov Theorem the PS condition is
satisfied. Hence, the functional $\mathcal{F}_{\gamma}$ satisfies the PS condition. Moreover,
by Lemma \ref{lezero} the functional $\mathcal{F}_{\gamma}$ possesses the MP geometry. Therefore,
the hypotheses of the Mountain Pass Theorem are fulfilled and we conclude that the functional
$\mathcal{F}_{\gamma}$ possesses a critical point $u\in H_0^1(\Omega)$. Moreover, if we define the set of paths
\begin{equation*}
\Gamma:=\{g\in C([0,1],H_0^1(\Omega))\,;\, g(0)=0,\; g(1)=\hat u\},
\end{equation*}
with $\hat u$ given as in the proof of Lemma \ref{lezero}, then,
\begin{equation*}
\mathcal{F}_{\gamma}(u)=c:=\inf_{g\in\Gamma} \max_{\theta \in [0,1]} \mathcal{F}_{\gamma}(g(\theta)).
\end{equation*}
To show that $u>0$, let us consider the functional,
\begin{equation*}
\mathcal{F}_{\gamma}^+(u)=\mathcal{F}_{\gamma}(u^+),
\end{equation*}
where $u^+=\max\{u,0\}$. Repeating with minor changes the arguments carried out above,
one readily shows that what was proved for the functional $\mathcal{F}_{\gamma}$ still holds for
the functional $\mathcal{F}_{\gamma}^+$. Therefore, $u\geq0$ and by the Maximum Principle, $u>0$.
\end{proof}
\begin{remark}
Assuming that $\partial\Omega$ is a $\mathcal{C}^2$ manifold, by standard elliptic regularity theory, \cite[Sec. 8.3, Theorem 1]{EV},
it follows that $u\in H_0^1(\Omega)\cap H^2(\Omega)$ and thus, $u$ is a positive weak solution to problem \eqref{eq:CH}.
\end{remark}

\subsection{Concentration-Compactness for the non-local problem \eqref{eq:nonlocal}.}\label{noexp2}


\noindent In this subsection we focus on the critical exponent case, $p=2^*-1$,
and our aim is to prove the PS condition for the functional $\mathcal{F}_{\gamma}$.
We carry out this task by means of a concentration-compactness argument based
on the following.
\begin{lemma}[P.-L. Lions,\cite{Lions}]\label{CC}
Let $\{u_n\}$ be a weakly convergent sequence to $u$ in $H_0^1(\Omega)$. Let $\mu$,
and $\nu$ be two nonnegative measures such that
\begin{equation*}
|\nabla u_n|^2\to\mu\quad\mbox{and}\quad|u_n|^{2^*}\to\nu\quad\mbox{as\ }n\to\infty.
\end{equation*}
Then, there exist a countable set $I$ of points $\{x_j\}_{j\in I}\subset \overline \Omega$
and some positive numbers $\mu_j$, and $\nu_j$ such that
\begin{equation}\label{deltas2}
\begin{split}
 |\nabla u_n|^2 \rightharpoonup \mu &=|\nabla u_0|^2+ \sum_{j\in I} \mu_j \delta_{x_j},\\
 |u_n|^{2^*} \rightharpoonup  \nu &=|u_0|^{2^*}+ \sum_{j\in I} \nu_j \delta_{x_j},
\end{split}
\end{equation}
where $\delta_{x_j}$ is the Dirac's delta centered at $x_j$ and satisfying
\begin{equation}\label{sobolev2}
\mu_j\geq S_N \nu_j^{2/2^*}.
\end{equation}
\end{lemma}

\begin{lemma}\label{PScondition_eq}
Assume $p=2^*-1$. Then, the functional $\mathcal{F}_{\gamma}$ satisfies the Palais-Smale
condition for any level $c$ such that,
\begin{equation*}
c<c^*=\frac{1}{N} S_N^{N/2}.
\end{equation*}
\end{lemma}
\begin{proof}
Although the proof is rather standard we include the details for the sake of completeness.
Let $\{u_n\}\subset H_0^1(\Omega)$ be a PS sequence of level $c<c^*$ for the functional
$\mathcal{F}_{\gamma}$. Thanks to Lemma \ref{acotacion_eq}, the sequence $\{u_n\}$ is uniformly
bounded and, as a consequence, we can assume that, up to a subsequence,
\begin{align}\label{conv:PS}
u_n \rightharpoonup u_0& \quad \hbox{weakly in}\quad H_0^1(\Omega),\notag\\
u_n \to u_0&\quad \hbox{strongly in}\quad L^q(\Omega), 1\leq q<2^*,\\
u_n \to u_0&\quad \hbox{a.e. in}\quad\Omega.\notag
\end{align}
Next, for $j\in I$ and $\varepsilon>0$, let $\varphi_{j,\varepsilon}\in C_0^\infty(\Omega)$ be a cut-off function such that,
\begin{equation}\label{cutoff2}
\varphi_{j,\varepsilon}=1 \quad \hbox{in}\quad B_{\varepsilon}(x_j),\quad \varphi_{j,\varepsilon}=0 \quad \hbox{in}
\quad B_{2\varepsilon}^c(x_j)\quad \hbox{and}\quad\displaystyle|\nabla \varphi_{j,\varepsilon}|\leq \frac{2}{\varepsilon},
\end{equation}
where $B_r(x_j)$ is the ball of radius $r>0$, centered at a point $x_j\in\overline{\Omega}$. Thus,
using $\varphi_{j,\varepsilon} u_n$ as a test function we find that,
\begin{align*}
\langle\mathcal{F}_{\gamma}'(u_n)|\varphi_{j,\varepsilon} u_n\rangle & =\int_{\Omega}\nabla u_n\cdot\nabla (\varphi_{j,\varepsilon} u_n)dx
-\gamma\int_{\Omega} \varphi_{j,\varepsilon} u_n (-\Delta)^{-1}u_ndx - \int_{\Omega} \varphi_{j,\varepsilon} |u_n|^{2^*}dx\\
&=\int_{\Omega} \varphi_{j,\varepsilon} |\nabla u_n|^2dx -\int_{\Omega} \varphi_{j,\varepsilon} |u_n|^{2^*}dx\\
&+\int_{\Omega} u_n \nabla u_n\cdot\nabla \varphi_{j,\varepsilon} dx -\gamma\int_{\Omega} \varphi_{j,\varepsilon} u_n (-\Delta)^{-1}u_ndx.
\end{align*}
Moreover, due to \eqref{deltas2} and \eqref{conv:PS},
\begin{equation*}
\lim_{n \to \infty}\langle \mathcal{F}_{\gamma}'(u_n)|\varphi_{j,\varepsilon}u_n\rangle =
\int_{\Omega} \varphi_{j,\varepsilon} d\mu-\int_{\Omega} \varphi_{j,\varepsilon} d\nu
-\gamma\int_{\Omega} \varphi_{j,\varepsilon} u_0 (-\Delta)^{-1}u_0dx
+\int_{\Omega} u_0 \nabla u_0\cdot\nabla\varphi_{j,\varepsilon}dx.
\end{equation*}
By construction,
\begin{equation*}
\lim_{\varepsilon \to 0} \left[-\gamma\int_{\Omega} \varphi_{j,\varepsilon} u_0 (-\Delta)^{-1}u_0dx
+\int_{\Omega} u_0\nabla u_0\cdot\nabla \varphi_{j,\varepsilon}dx\right]=0.
\end{equation*}
Then, as $\mathcal{F}_{\gamma}'(u_n)\to0$ in $\left(H_0^1(\Omega)\right)'$, we obtain that,
\begin{equation*}
\lim_{\varepsilon \to 0} \left(\int_{\Omega} \varphi_{j,\varepsilon} d\mu-\int_{\Omega} \varphi_{j,\varepsilon} d\nu\right)
=\mu_j-\nu_j=0,
\end{equation*}
and we conclude,
\begin{equation}\label{cons:concentration}
\nu_j = \mu_j.
\end{equation}
Finally, we have two options either the PS sequence has a convergent subsequence or it concentrates
around some of the points $x_j$. In other words, $\nu_j=\mu_j=0$, or there exists some $\nu_j>0$ such
that, by \eqref{sobolev2} and \eqref{cons:concentration}, $\nu_j \geq S_N^{N/2}$. In case of having
concentration, we find that
\begin{align*}
c&=\lim_{n \to \infty} \mathcal{F}_{\gamma}(u_n)=\lim_{n \to \infty} \mathcal{F}_{\gamma}(u_n)-\frac{1}{2}\langle\mathcal{F}_{\gamma}'(u_n)| u_n\rangle\\
&=\left(\frac{1}{2}-\frac{1}{2^*}\right)\int_{\Omega}|u_0|^{2^*}dx+\left(\frac{1}{2}-\frac{1}{2^*}\right)\nu_j\\
&\geq \frac{1}{N}S_N^{N/2}=c^*,
\end{align*}
in contradiction with the hypotheses $c<c^*$. Therefore, the PS sequence has a convergent subsequence and the PS condition is satisfied.
\end{proof}
It remains to show that we can obtain a path for $\mathcal{F}_{\gamma}$ under the critical
level $c^*$. In order to get such path we will take test functions of the form
\begin{equation*}
\tilde u_{\varepsilon}=M\phi_{\varepsilon},
\end{equation*}
where
\begin{equation}\label{phie2}
\phi_{\varepsilon}=\varphi_{j,R}\; u_{j,\varepsilon},
\end{equation}
with $\varphi_{j,R}$ a cut-off function defined as \eqref{cutoff2} for some $R>0$
small enough, $M>0$ a large enough constant such that $\mathcal{F}_{\gamma}(\tilde u_{\varepsilon})<0$
and $u_{j,\varepsilon}$ are the family of functions
\begin{equation}\label{extrem}
u_{j,\varepsilon}(x)=\left(\frac{\varepsilon}{\varepsilon^2+|x-x_j|^2}\right)^{\frac{N-2}{2}},
\end{equation}
for $\varepsilon>0$. Let us notice that the functions $u_{j,\varepsilon}$ are the extremal functions for the Sobolev's inequality in $\mathbb{R}^N$,
where the constant $S_N$ is achieved (see \cite{Ta}). Then,
\begin{equation*}
\int_{\mathbb{R}^N} |\nabla u_{j,\varepsilon}|^2 dx= S_N \left(\int_{\mathbb{R}^N} |u_{j,\varepsilon}|^{p+1}dx\right)^{2/2^*}.
\end{equation*}
For the sake of simplicity we will consider $x_j=0$, we will denote
$\varphi_{j,R}=\varphi$ under the construction \eqref{cutoff2} and $u_{j,\varepsilon}=u_{\varepsilon}$. We will also assume the normalization
\begin{equation}\label{norms2}
\|u_{\varepsilon}\|_{L^{2^*}(\Omega)}=1,
\end{equation}
so that the Sobolev constant is given by
\begin{equation*}
S_N=\int_{\mathbb{R}^N} |\nabla u_{\varepsilon}|^2dx.
\end{equation*}
Then, under the previous considerations we define the set of paths
\begin{equation*}
\Gamma_{\varepsilon}:=\{g\in C([0,1],H_0^1(\Omega))\,;\, g(0)=0,\; g(1)=\tilde u_{\varepsilon}\},
\end{equation*}
and we consider the minimax values
\begin{equation*}
c_{\varepsilon}=\inf_{g\in\Gamma_{\varepsilon}} \max_{t \in [0,1]} \mathcal{F}_{\gamma}(g(t)).
\end{equation*}
The final issue we must solve now is the fact that the levels $c_{\varepsilon}$ are
always below $c^*$ for $\varepsilon$ small enough. To that end, we recall the following.
\begin{lemma}[\cite{BN}, Lemma 1.1]\label{lees}
Let $\phi$ be the function denoted by \eqref{phie2} around the point $x_j=0$. Then,
\begin{equation}\label{esl2}
\int_{\mathbb{R}^N} \phi_{\varepsilon}^2dx
=\left\{\begin{array}{ll}
 C\varepsilon+O(\varepsilon^2) & \hbox{if}\quad N=3,\\
\frac{C\varepsilon^2}{2} |\log \varepsilon|+O(\varepsilon^2) & \hbox{if}\quad N=4,\\
C\varepsilon^2+O(\varepsilon^{N-2}) & \hbox{if}\quad N\geq 5.
\end{array}\right.
\end{equation}
Moreover,
\begin{equation}\label{esn2}
\|\nabla \phi_{\varepsilon}\|_2^2 =S_N+ O(\varepsilon^{N-2}).
\end{equation}
\end{lemma}

\begin{remark}
Using similar arguments one could also estimate $\|\phi_{\varepsilon}\|_{L^{2^*}(\Omega)}\sim C$ however,
it is simpler if we normalize it as done in \eqref{norms2}.
\end{remark}
\noindent To carry out the analysis of the levels $c_{\varepsilon}$
we need estimates dealing with the following  term
$\int_{\Omega}\phi_{\varepsilon}(-\Delta)^{-1}\phi_{\varepsilon}
dx$. To do so, we prove the following.
\begin{lemma}\label{lem:estimacion_nolocal}
\label{lees22}
Let $\phi_{\varepsilon}$ be the function denoted by \eqref{phie2} around the point $x_j=0$. Then,
there exists a constant $C>0$ independent of $\varepsilon$ such that
\begin{equation}\label{radial2}
\int_{\Omega}\phi_{\varepsilon}(-\Delta)^{-1}\phi_{\varepsilon}dx> C\varepsilon^{4}\quad \mbox{if}\ N=6,
\end{equation}
\begin{equation}\label{radial22}
\int_{\Omega}\phi_{\varepsilon}(-\Delta)^{-1}\phi_{\varepsilon}dx> C\varepsilon^{\mu}\quad \mbox{if}\ N\geq7,
\end{equation}
where $\frac{N}{2}+1>\mu>1+\frac{N}{N-4}$.
\end{lemma}
\begin{proof}
Let $v_{\varepsilon}(x)=(-\Delta)^{-1}\phi_{\varepsilon}(x)$ and note that because of the definition of the
cut-off function \eqref{cutoff2}, we can choose $v_{\varepsilon}(x)$ such that
\begin{equation*}
        \left\{
        \begin{tabular}{lcl}
        $(-\Delta) v_{\varepsilon}=\phi_{\varepsilon}$ & &in $B_{2R}(0)$, \\
        $v_{\varepsilon}=0$                   & &in $\partial B_{2R}(0).$
        \end{tabular}
        \right.
\end{equation*}
Moreover, since $\phi_{\varepsilon}>0$ in $B_{2R}(0)$, thanks to the Maximum Principle, it follows
that $v_\varepsilon>0$ in $B_{2R}(0)$. Now, let us notice that for any $x\in B_{R}(0)$
we have $\phi_{\varepsilon}(x)=u_{\varepsilon}(x)$ as well as
\begin{equation*}
\frac{\varepsilon^{-\frac{N-2}{2}}}{\left(1+\left(\frac{R}{\varepsilon}\right)^2\right)^{\frac{N-2}{2}}}
\leq u_{\varepsilon}(x)\leq \varepsilon^{-\frac{N-2}{2}}.
\end{equation*}
Next, take $\rho<\frac{R}{2}$ and consider the function
$\widetilde{v}(x)=\frac{2}{N}\left(1-\left(\frac{|x|}{2\rho}\right)^2\right)_+$, where
$(\cdot)_+$ stands for the positive part. Then, $\widetilde{v}$ satisfies the problem
\begin{equation*}
        \left\{
        \begin{tabular}{lcl}
        $(-\Delta) \widetilde{v}=\frac{1}{\rho^2}$ & &in $B_{2\rho}(0)$, \\
        $\widetilde{v}=0$                            & &in $\partial B_{2\rho}(0).$
        \end{tabular}
        \right.
\end{equation*}
To apply a comparison principle we choose $\rho=\varepsilon^{\alpha}$, with $\alpha>0$, such that
\begin{equation*}
(-\Delta) \widetilde{v}\leq(-\Delta)v_{\varepsilon}\quad\quad \mbox{in}\quad B_{2\rho}(0).
\end{equation*}
Then, given $\varepsilon>0$ arbitrarily small, we distinguish two cases depending upon $\alpha\geq1$
or $\alpha<1$. In the first case, since
\begin{equation*}
u_{\varepsilon}(x)\bigg|_{x\in B_{2\rho}(0)}
\geq \frac{\varepsilon^{-\frac{N-2}{2}}}{\left(1+\left(\frac{2\rho}{\varepsilon}\right)^2\right)^{\frac{N-2}{2}}}
=\frac{\varepsilon^{-\frac{N-2}{2}}}{\left(1+4\varepsilon^{2(\alpha-1)}\right)^{\frac{N-2}{2}}}\geq c_1\varepsilon^{-\frac{N-2}{2}},
\end{equation*}
for a positive constant $c_1<1$, we need to choose $\alpha$ such that,
\begin{equation*}
\frac{1}{\varepsilon^{2\alpha}}\leq c_1\varepsilon^{-\frac{N-2}{2}}.
\end{equation*}
We conclude $2\alpha\leq\frac{N-2}{2}$. Therefore, we obtain the range $1\leq\alpha\leq\frac{N-2}{4}$,
which necessarily requires $N\geq6$. In the second case, $\alpha<1$, since
\begin{equation*}
u_{\varepsilon}(x)\bigg|_{x\in B_{2\rho}(0)}
\geq\frac{\varepsilon^{-\frac{N-2}{2}}}{\left(1+4\varepsilon^{-2(1-\alpha)}\right)^{\frac{N-2}{2}}}
\geq c_2\varepsilon^{-\frac{N-2}{2}+(1-\alpha)(N-2)},
\end{equation*}
for a positive constant $c_2<\frac{1}{4}$, we need to choose $\alpha$ such that
\begin{equation*}
\frac{1}{\varepsilon^{2\alpha}}\leq c_2\varepsilon^{-\frac{N-2}{2}+(1-\alpha)(N-2)}.
\end{equation*}
Then, we obtain the condition $\alpha\geq\frac{1}{2}+\frac{1}{N-4}$ that,
together with $\alpha<1$, implies $N>6$. Finally, by construction,
\begin{equation*}
0=\widetilde{v}(x)\bigg|_{x\in \partial B_{2\rho}(0)}< v_{\varepsilon}(x)\bigg|_{x\in \partial B_{2\rho}(0)}
\end{equation*}
Because of the Maximum Principle, we conclude that $v_{\varepsilon}(x)>\widetilde{v}(x)$ for $x\in B_{2\rho}(0)$ thus,
\begin{align*}
\int_{\Omega}\phi_{\varepsilon}(-\Delta)^{-1}\phi_{\varepsilon} dx
&\geq\int_{B_{R}(0)}u_{\varepsilon}(x)v_{\varepsilon}(x)dx>\int_{B_{2\rho}(0)}u_{\varepsilon}(x)\widetilde{v}(x)dx\\
&\geq\int_{B_{\rho}(0)}u_{\varepsilon}(x)\widetilde{v}(x)dx=\frac{2}{N}\int_{B_{\rho}(0)}u_{\varepsilon}(x)\left(1-\left(\frac{|x|}{2\rho}\right)^2\right)dx\\
&\geq \frac{3}{2N}\int_{B_{\rho}(0)}u_{\varepsilon}(x)dx.
\end{align*}
On the other hand,
\begin{align*}
\int_{B_{\rho}(0)}u_{\varepsilon}(x)dx
&=\varepsilon^{-\frac{N-2}{2}}\int_{B_{\rho}(0)}\frac{1}{\left(1+\left(\frac{|x|}{\varepsilon}\right)^2\right)^{\frac{N-2}{2}}}dx
=\varepsilon^{-\frac{N-2}{2}}\int_0^{\rho}\frac{r^{N-1}}{\left(1+\left(\frac{r}{\varepsilon}\right)^2\right)^{\frac{N-2}{2}}}dr\\
&=\varepsilon^{-\frac{N-2}{2}+N-1}\int_0^{\rho}\frac{\left(r/\varepsilon\right)^{N-1}}{\left(1+\left(\frac{r}{\varepsilon}\right)^2\right)^{\frac{N-2}{2}}}dr
=\varepsilon^{\frac{N}{2}+1}\int_0^{\rho/\varepsilon}\frac{s^{N-1}}{\left(1+s^2\right)^{\frac{N-2}{2}}}ds\\
&\geq c\varepsilon^{\frac{N}{2}+1}\int_0^{\rho/\varepsilon}s^{N-1}ds=c\varepsilon^{\frac{N}{2}+1}\left(\frac{\rho}{\varepsilon}\right)^N,
\end{align*}
for a positive constant $c$. Then, since we have chosen $\rho=\varepsilon^{\alpha}$, we obtain
\begin{equation}\label{a1}
\int_{\Omega}\phi_{\varepsilon}(-\Delta)^{-1}\phi_{\varepsilon} dx
> C\varepsilon^{\frac{N}{2}+1+N(\alpha-1)}\quad \mbox{for}\ \alpha\geq1, N\geq6,
\end{equation}
and
\begin{equation}\label{a2}
\int_{\Omega}\phi_{\varepsilon}(-\Delta)^{-1}\phi_{\varepsilon} dx
> C\varepsilon^{\frac{N}{2}+1-N(1-\alpha)}\quad \mbox{for}\ 1>\alpha>\frac{1}{2}+\frac{1}{N-4}, N\geq7.
\end{equation}
Now, we note that for the range $\alpha\geq1$ the value $\alpha=1$ provides us with
the optimum estimate in \eqref{a1} and, thus, from here we obtain
\begin{equation}\label{a3}
\int_{\Omega}\phi_{\varepsilon}(-\Delta)^{-1}\phi_{\varepsilon} dx
> C\varepsilon^{\frac{N}{2}+1}\quad \mbox{for } N\geq6.
\end{equation}
Moreover, since $\frac{N}{2}+1>\frac{N}{2}+1-N(1-\alpha)$ for $1>\alpha>\frac{1}{2}+\frac{1}{N-4}$,
inequality \eqref{a2} provides a stronger bound than the one provided by inequality \eqref{a3}
for any $N\geq7$. Thus, inequality \eqref{a3} is only useful for $N=6$, from where we conclude
\eqref{radial2}. Finally, setting $\mu=\frac{N}{2}+1-N(1-\alpha)$ in \eqref{a2}, it follows that
$\frac{N}{2}+1>\mu>1+\frac{N}{N-4}$, and we conclude \eqref{radial22}.
\end{proof}

Next we perform the analysis of the levels $c_{\varepsilon}$, proving that, in fact, the levels
$c_{\varepsilon}$ are always below the critical level $c^*$ provided $\varepsilon>0$ is small enough.

\begin{lemma}\label{lemma_inverse}
Assume $p=2^*-1$ and $N\geq7$. Then, there exists $\varepsilon>0$ small enough such that,
\begin{equation*}
\sup_{0\leq t\leq1} \mathcal{F}_{\gamma}(t\tilde u_{\varepsilon})< \frac{1}{N} S_N^{N/2}.
\end{equation*}
\end{lemma}
\begin{proof}
Using \eqref{esn2} in Lemma \ref{lees} and assuming the normalization \eqref{norms2}, we find
\begin{align*}
g(t):=\mathcal{F}_{\gamma}(t\tilde u_{\varepsilon})
&=\frac{t^2 M^2}{2}\|\nabla \phi_{\varepsilon}\|_{L^2(\Omega)}^2
-\frac{t^2 M^2 \gamma}{2}\int_{\Omega}\phi_{\varepsilon}(-\Delta)^{-1}\phi_{\varepsilon} dx
-\frac{t^{2^*} M^{2^*}}{2^*}\\
&=\frac{M^2}{2}\left( S_N+O(\varepsilon^{N-2})-\gamma F(\varepsilon)\right)t^2-\frac{M^{2^*}}{2^*}t^{2^*},
\end{align*}
where $F(\varepsilon)=\int_{\Omega}\phi_{\varepsilon}(-\Delta)^{-1}\phi_{\varepsilon}dx$. It is clear that
$\displaystyle \lim_{t\to \infty} g(t)=-\infty$ as well as that $g(t)>0$ for $t>0$ small
enough, therefore, the function $g(t)$ possesses a maximum value at the point,
\begin{equation*}
t_{\varepsilon}:=\left(\frac{M^2\left( S_N+O(\varepsilon^{N-2})-\gamma F(\varepsilon)\right)}{M^{2^*}}\right)^{\frac{1}{2^*-2}}.
\end{equation*}
Moreover, at this point $t_{\varepsilon}$ we have,
\begin{equation*}
g(t_{\varepsilon})=\frac{1}{N}\left(S_N+O(\varepsilon^{N-2})-\gamma F(\varepsilon)\right)^{N/2}.
\end{equation*}
Then, the proof will be completed if the inequality
\begin{equation*}
\frac{1}{N}\left(S_N+O(\varepsilon^{N-2})-\gamma F(\varepsilon)\right)^{N/2}<\frac{1}{N}S_N^{N/2},
\end{equation*}
or, equivalently, the inequality
\begin{equation}\label{eq1}
O(\varepsilon^{N-2})<\gamma F(\varepsilon),
\end{equation}
holds true provided $\varepsilon$ is small enough. Moreover, because of \eqref{radial22} in
Lemma \ref{lees22}, we have that $F(\varepsilon)>C\varepsilon^{\mu}$ with
$\frac{N}{2}+1>\mu>1+\frac{N}{N-4}$. To finish the proof, let us show that, in fact,
the stronger inequality
\begin{equation}\label{ultt}
O(\varepsilon^{N-2})<C\varepsilon^{\mu},
\end{equation}
holds true provided $\varepsilon$ is small enough. To that end is enough to observe that \eqref{ultt} requires
 $N-2>\mu$ that, together $\frac{N}{2}+1>\mu>1+\frac{N}{N-4}$, provides us with
 the condition $1+\frac{N}{N-4}<N-2$ which is equivalent to $(N-2)(N-6)>0$,
that is obviously satisfied. Thus, inequality \eqref{eq1} is satisfied provided $\varepsilon$ is small enough.
\end{proof}

\begin{remark}
In the proof of Lemma \ref{lemma_inverse} we proved that, for $N\geq7$,
$O(\varepsilon^{N-2})<C\varepsilon^{\mu}$ provided $\varepsilon$ is small enough and, because of \eqref{radial22} in
Lemma \ref{lees22}, we concluded $O(\varepsilon^{N-2})<C\varepsilon^{\mu}<F(\varepsilon)$. If we take $N=6$ and we
repeat the steps above, we readily find that \eqref{radial2} in Lemma \ref{lees22} lead us
to prove $O(\varepsilon^4)<C\varepsilon^4$, that can not be ensured either $\varepsilon>0$ arbitrarily small or not.
As we will see below (see Lemma \ref{level}), this restriction on the dimension is
not a merely consequence of the accuracy of the estimates in Lemma \ref{lees22}.
\end{remark}

\begin{proof}[Proof of Theorem \ref{Thequation}.]\hfill\break
Thanks to Lemma \ref{lezero} and Lemma \ref{lemma_inverse}, we find that
\begin{equation*}
0<c_{\varepsilon}\leq \sup_{0\leq t\leq1} \mathcal{F}_{\gamma}(t\tilde u_{\varepsilon})< \frac{1}{N} S_N^{N/2},
\end{equation*}
provided $\varepsilon>0$ is small enough. Because of Lemma \ref{lezero} the functional $\mathcal{F}_{\gamma}$ has the
MP geometry. Moreover, because of Lemma \ref{PScondition_eq} the functional $\mathcal{F}_{\gamma}$ satisfies
the PS condition for any level $c_{\varepsilon}$ provided $\varepsilon>0$ is small enough. Therefore, we can apply the
Mountain Pass Theorem to obtain the existence of a critical point $u\in H_0^1(\Omega)$. The rest
follows as in the subcritical case.
\end{proof}
\section{Existence of positive solutions for the system \eqref{system}}\label{concomp2}


In this section we provide the existence result for the system \eqref{system}. We start by
stating the analogous results of those obtained for the functional $\mathcal{F}_{\gamma}$.
\begin{lemma}\label{lezero_sys}
The functional $\mathcal{J}_{\gamma}$ denoted by \eqref{fn:system} has the MP geometry.
\end{lemma}
\begin{proof}
Let us consider, without loss of generality, a pair $(g,h)\in H_0^1(\Omega)\times H_0^1(\Omega)$
such that $\|g\|_{L^{p+1}(\Omega)}=1$. Then, taking a real number $t>0$ and using the Young's
inequality together with the Poincar\'e inequality and the Sobolev inequality \eqref{sobolev},
we find,
\begin{equation}\label{mini}
\begin{split}
\mathcal{J}_{\gamma}(tg,th)=
&\frac{t^2}{2}\int_{\Omega} |\nabla g|^2dx + \frac{t^2}{2} \int_{\Omega} |\nabla h|^2dx
 -t^2\sqrt{\gamma}\int_{\Omega} gh\;dx-\frac{t^{p+1}}{p+1}\\
\geq&\frac{t^2}{2} \left(\|g\|_{H_0^1(\Omega)}^2+\|h\|_{H_0^1(\Omega)}^2
    -\sqrt{\gamma}\int_{\Omega}g^2dx-\sqrt{\gamma}\int_{\Omega}h^2dx\right)-\frac{t^{p+1}}{p+1}\\
\geq&\frac{t^2}{2}\left(1-\frac{\sqrt{\gamma}}{\lambda_1}\right)\left(\|g\|_{H_0^1(\Omega)}^2+\|h\|_{H_0^1(\Omega)}^2\right)
    -\|g\|_{H_0^1(\Omega)}^2\frac{C}{p+1}t^{p+1}\\
\geq&\left(\frac{1}{2}\left(1-\frac{\sqrt{\gamma}}{\lambda_1}\right)t^2-\frac{C}{p+1}t^{p+1}\right)
     \left(\|g\|_{H_0^1(\Omega)}^2+\|h\|_{H_0^1(\Omega)}^2\right),
\end{split}
\end{equation}
where $\lambda_1$ is the first eigenvalue of the Laplace operator under Dirichlet
boundary conditions. Since $0<\gamma<\lambda_1^*=\lambda_1^2$ it follows that $\sqrt{\gamma}<\lambda_1$
and we obtain $\left(1-\frac{\sqrt{\gamma}}{\lambda_1}\right)>0$. Therefore,
taking $t>0$ such that,
\begin{equation*}
0<t^{p-1}<\frac{p+1}{2C}\left(1-\frac{\sqrt{\gamma}}{\lambda_1}\right),
\end{equation*}
from \eqref{mini} we conclude
\begin{equation*}
\mathcal{J}_{\gamma}(tg,th)>0.
\end{equation*}
Thus, the functional $\mathcal{J}_{\gamma}$ has a local minimum at $(u,v)=(0,0)$, i.e.,
\begin{equation*}
\mathcal{J}_{\gamma}(tg,th)>\mathcal{J}_{\gamma}(0,0)=0,
\end{equation*}
for any pair $(g,h)\in H_0^1(\Omega)\times H_0^1(\Omega)$ provided $t>0$ is small enough.
Also, it is clear that, because of the Poincar\'e inequality,
\begin{equation*}
\begin{split}
\mathcal{J}_{\gamma}(tg,th)=
&\frac{t^2}{2}\int_{\Omega}|\nabla g|^2dx + \frac{t^2}{2}\int_{\Omega}|\nabla h|^2dx
 -t^2\sqrt{\gamma}\int_{\Omega}gh\;dx-\frac{t^{p+1}}{p+1}\\
\leq&\frac{t^2}{2} \left(\|g\|_{H_0^1(\Omega)}^2+\|h\|_{H_0^1(\Omega)}^2
    +\sqrt{\gamma}\int_{\Omega}g^2dx+\sqrt{\gamma}\int_{\Omega}h^2dx\right)-\frac{t^{p+1}}{p+1}\\
\leq&\frac{t^2}{2}\left(1+\frac{\sqrt{\gamma}}{\lambda_1}\right)\left(\|g\|_{H_0^1(\Omega)}^2+\|h\|_{H_0^1(\Omega)}^2\right)
    -\frac{t^{p+1}}{p+1}.
\end{split}
\end{equation*}
Then,
\begin{equation*}
\mathcal{J}_{\gamma}(tg,th)\to-\infty,\quad\mbox{as }t\to\infty,
\end{equation*}
and thus, there exists a pair $(\hat u,\hat v)$ such that $\mathcal{J}_{\gamma}(\hat u,\hat v)<0$.
\end{proof}
\begin{lemma}\label{acotacion_sys}
Let $\{(u_n,v_n)\}\subset H_0^1(\Omega)\times H_0^1(\Omega)$ be a PS sequence at level $c$
for the functional $\mathcal{J}_{\gamma}$, i.e.
\begin{equation*}
\mathcal{J}_{\gamma}(u_n,v_n) \rightarrow c,\quad\mathcal{J}_{\gamma}'(u_n,v_n) \rightarrow 0,\quad \hbox{as}\quad n\to \infty.
\end{equation*}
Then,
\begin{equation*}
\{(u_n,v_n)\}\quad \hbox{is bounded in}\quad H_0^1(\Omega)\times H_0^1(\Omega).
\end{equation*}
\end{lemma}
\begin{proof}
Since $\mathcal{J}_{\gamma}'(u_n,v_n) \rightarrow 0$ in $\left(H_0^1(\Omega)\times H_0^1(\Omega)\right)'$, in particular
\begin{equation*}
\left\langle\mathcal{J}_{\gamma}'(u_n,v_n)|\frac{(u_n,v_n)}{\|u_n\|_{H_0^1(\Omega)}+\|v_n\|_{H_0^1(\Omega)}}\right\rangle\to 0.
\end{equation*}
Thus, for any $\varepsilon>0$, there exists a subsequence, denoted again by $\{(u_n,v_n)\}$, such that,
\begin{equation*}
\int_{\Omega} |\nabla u_n|^2dx +\int_{\Omega} |\nabla v_n|^2dx -2\sqrt{\gamma}\int_{\Omega} u_nv_ndx
 -\int_{\Omega} |u_n|^{p+1}dx=\left[\|u_n\|_{H_0^1(\Omega)}+\|v_n\|_{H_0^1(\Omega)}\right]\cdot o(1).
\end{equation*}
Moreover, since $\mathcal{J}_{\gamma}(u_n,v_n) \to c$,
\begin{equation*}
\frac{1}{2} \int_{\Omega} |\nabla u_n|^2dx +\frac{1}{2}\int_{\Omega} |\nabla v_n|^2dx
-\sqrt{\gamma}\int_{\Omega} u_nv_ndx-\frac{1}{p+1} \int_{\Omega} |u_n|^{p+1}dx=c+o(1),
\end{equation*}
for $n>0$ big enough. Therefore, for a positive constant $\mu$ (to be determined below)
we find that
\begin{equation*}
\mathcal{J}_{\gamma}(u_n,v_n)-\mu \left\langle \mathcal{J}_{\gamma}'(u_n,v_n)|\frac{1}{\|u_n\|_{H_0^1(\Omega)}}(u_n,v_n)\right\rangle
 =c+\left[\|u_n\|_{H_0^1(\Omega)}+\|v_n\|_{H_0^1(\Omega)}\right]\cdot o(1).
\end{equation*}
That is,
\begin{align*}
&\left(\frac{1}{2}-\mu\right)\left[\int_{\Omega}|\nabla u_n|^2dx+\int_{\Omega} |\nabla v_n|^2dx\right]
 -(1-2\mu)\sqrt{\gamma}\int_{\Omega} u_nv_ndx-\left(\frac{1}{p+1}-\mu\right) \int_{\Omega}|u_n|^{p+1}dx\\
&=c+\left[\|u_n\|_{H_0^1(\Omega)}+\|v_n\|_{H_0^1(\Omega)}\right]\cdot o(1).
\end{align*}
Hence, taking $\mu$ such that $\frac{1}{p+1}<\mu<\frac{1}{2}$,
\begin{equation*}
\left(\frac{1}{2}-\mu\right)\!\!\left[\int_{\Omega}|\nabla u_n|^2dx\!+\!\!\int_{\Omega} |\nabla v_n|^2dx\right]
-(1-2\mu)\sqrt{\gamma}\!\!\int_{\Omega} u_nv_ndx\leq c+\left[\|u_n\|_{H_0^1(\Omega)}+\|v_n\|_{H_0^1(\Omega)}\right]\cdot o(1),
\end{equation*}
and using Young's inequality,
\begin{align*}
&\left(\frac{1}{2}-\mu\right)\left[\int_{\Omega}|\nabla u_n|^2dx+\int_{\Omega} |\nabla v_n|^2dx
-\sqrt{\gamma} \int_{\Omega} u_n^2dx-\sqrt{\gamma} \int_{\Omega} v_n^2dx\right]\\
&\leq c+\left[\|u_n\|_{H_0^1(\Omega)}+\|v_n\|_{H_0^1(\Omega)}\right]\cdot o(1).
\end{align*}
Then, because of the Poincar\'e inequality, we conclude
\begin{equation}\label{eq_ps1}
\left(\frac{1}{2}-\mu\right)\left(1-\frac{\sqrt{\gamma}}{\lambda_1}\right)\left[ \|u_n\|_{H_0^1(\Omega)}^2 +\|v_n\|_{H_0^1(\Omega)}^2 \right]
   \leq c+\left[\|u_n\|_{H_0^1(\Omega)}+\|v_n\|_{H_0^1(\Omega)}\right]\cdot o(1),
\end{equation}
where $\lambda_1$ is the first eigenvalue of the Laplace operator under Dirichlet
boundary conditions. Since $0<\gamma<\lambda_1^*=\lambda_1^2$, it follows that
\begin{equation*}
\left(\frac{1}{2}-\mu\right)\left(1-\frac{\sqrt{\gamma}}{\lambda_1}\right)>0,
\end{equation*}
and thus, by \eqref{eq_ps1}, we conclude that the sequence  $\{(u_n,v_n)\}$ is bounded
in $H_0^1(\Omega)\times H_0^1(\Omega)$.
\end{proof}

\begin{proof}[Proof of Theorem \ref{Thsystem_subcritical}.]\hfill\break
If $1<p<2^*-1$, given a PS sequence $\{(u_n,v_n)\}\subset H_0^1(\Omega)\times H_0^1(\Omega)$ at level $c$,
by Lemma \ref{lezero_sys}, the functional $\mathcal{J}_{\gamma}$ has the MP geometry.
Moreover, by Lemma \ref{acotacion_sys} and the compact inclusion
\begin{equation*}
H_0^1(\Omega)\times H_0^1(\Omega) \subset\subset L^{p+1}(\Omega)\times L^{p+1}(\Omega),\quad\mbox{for } 2\leq p+1<2^*,
\end{equation*}
provided by Rellich-Kondrachov Theorem, the functional $\mathcal{J}_{\gamma}$ satisfies the PS
condition at any level $c$. Therefore, the hypotheses of the Mountain Pass Theorem are fulfilled and
we conclude that the functional $\mathcal{J}_{\gamma}$ possesses a critical point $(u,v)\in H_0^1(\Omega)\times H_0^1(\Omega)$. Moreover,
if we define the set of the paths
\begin{equation*}
\Gamma:=\left\{g\in C\left([0,1],H_0^1(\Omega)\times H_0^1(\Omega)\right)\,;\, g(0)=(0,0),\; g(1)=(\hat{u},\hat{v})\right\},
\end{equation*}
with $(\hat u,\hat v)$ given as in the proof of Lemma \ref{lezero_sys}, then
\begin{equation*}
\mathcal{J}_{\gamma}(u,v)=c:=\inf_{g\in\Gamma} \max_{\theta \in [0,1]} \mathcal{J}_{\gamma}(g(\theta)).
\end{equation*}
To show the positivity of the pair $(u,v)$ we argue as in the proof of Theorem \ref{Thequation_subcritical}.
Let us consider the functional,
\begin{equation*}
\mathcal{J}_{\gamma}^+(u,v)=\mathcal{J}_\gamma(u^+,v^+),
\end{equation*}
where, as before, $u^+=\max\{u,0\}$. Repeating with minor changes the arguments carried out above for
the functional $\mathcal{J}_{\gamma}$ we conclude that the functional $\mathcal{J}_{\gamma}^+$ has a critical point
$(\tilde{u},\tilde{v})$ such that $\tilde{u}\geq0$ and $\tilde{v}\geq0$. Moreover, by the Maximum
Principle, it follows that $\tilde{u}>0$ and $\tilde{v}>0$, then $(\tilde{u},\tilde{v})$ is a positive solution of \eqref{system}.
\end{proof}
To prove the PS condition when $p+1=2^*$ we must apply once again a concentration-compactness argument.

\begin{lemma}\label{PScondition}
Assume $p=2^*-1$. Then, the functional $\mathcal{J}_{\gamma}$ satisfies the Palais-Smale condition for any level $c$ such that,
\begin{equation*}
c<c^*=\frac{1}{N} S_N^{N/2}.
\end{equation*}
\end{lemma}

\begin{proof}
Let $\{(u_n,v_n)\}\subset H_0^1(\Omega)\times H_0^1(\Omega)$ be a PS sequence of level
$c<c^*$ for the functional $\mathcal{J}_{\gamma}$. Thanks to Lemma \ref{acotacion_sys},
the sequence $\{(u_n,v_n)\}$ is uniformly bounded and, as a consequence, we can
assume that there exists a subsequence still denoted by $\{(u_n,v_n)\}$, such that,
\begin{align}\label{conv:PS_system}
(u_n,v_n) \rightharpoonup (u_0,v_0)& \quad \hbox{weakly in}\quad H_0^1(\Omega)\times H_0^1(\Omega),\notag\\
(u_n,v_n) \to (u_0,v_0)&\quad \hbox{strongly in}\quad L^q(\Omega)\times L^q(\Omega), 1\leq q<2^*,\\
(u_n,v_n) \to (u_0,v_0)&\quad \hbox{a.e. in}\quad\Omega.\notag
\end{align}
Moreover, we can assume that, up to a subsequence, there exist three measures $\mu$, $\tilde\mu$ and
$\nu$ such that $|\nabla u_n|^2$, $|\nabla v_n|^2$ and $|u_n|^{2^*}$,
converge in the sense of the measures $\mu$, $\tilde\mu$ and $\nu$ respectively. Thus,
because of Lemma \ref{CC}, there is a countable set $I$ of points $\{x_j\}_{j\in I}\subset \overline{\Omega}$,
and some positive numbers $\mu_j$, $\tilde\mu_j$ and $\nu_j$ such that
\begin{equation}\label{licom}
\begin{split}
& |\nabla u_n|^2 \rightharpoonup d\mu =|\nabla u_0|^2+ \sum_{j\in I} \mu_j \delta_{x_j},\\ &
|\nabla v_n|^2  \rightharpoonup d\tilde\mu =|\nabla v_0|^2+ \sum_{j\in I} \tilde\mu_j \delta_{x_j},\\ &
|u_n|^{2^*} \rightharpoonup  d\nu =|u_0|^{2^*}+ \sum_{j\in I} \nu_j \delta_{x_j},
\end{split}
\end{equation}
where $\delta_{x_j}$ is the Dirac's delta centered at $x_j$ with $j\in I$ and satisfying
\begin{equation}\label{mescon}
\mu_j\geq S_N \nu_j^{2/2^*}.
\end{equation}
Next, for $j\in I$, let $\varphi_{j,\varepsilon}\in C_0^\infty(\Omega)$ be a cut-off function
satisfying \eqref{cutoff2} centered at $x_j\in\overline{\Omega}$. Thus, using
$(\varphi_{j,\varepsilon} u_n,\varphi_{j,\varepsilon} v_n)$ as a test function, we find,
\begin{align*}
\langle \mathcal{J}_{\gamma}'(u_n,v_n)|(\varphi_{j,\varepsilon} u_n,\varphi_{j,\varepsilon} v_n)\rangle\!=
 &\!\!\int_{\Omega}\nabla u_n \cdot \nabla (\varphi_{j,\varepsilon} u_n)dx\!
  +\!\int_{\Omega} \nabla v_n\cdot\nabla(\varphi_{j,\varepsilon} v_n)dx
    -2\sqrt{\gamma}\!\int_{\Omega} \varphi_{j,\varepsilon} u_n v_ndx\\
-&\int_{\Omega} \varphi_{j,\varepsilon} u_n^{2^*}dx\\
=&\int_{\Omega} \varphi_{j,\varepsilon} |\nabla u_n|^2dx
 +\int_{\Omega} \varphi_{j,\varepsilon} |\nabla v_n|^2dx
 -\int_{\Omega} \varphi_{j,\varepsilon} u_n^{2^*}dx\\
+&\!\!\int_{\Omega}\! u_n \langle \nabla u_n,\nabla \varphi_{j,\varepsilon}\rangle dx
  +\!\!\int_{\Omega}\! v_n \langle \nabla v_n,\nabla \varphi_{j,\varepsilon}\rangle dx
  -2\sqrt{\gamma}\!\!\int_{\Omega}\! \varphi_{j,\varepsilon} u_n v_n dx.
\end{align*}
Moreover, due to \eqref{conv:PS_system} and \eqref{licom},
\begin{align*}
\lim_{n\to \infty}\langle \mathcal{J}_{\gamma}'(u_n,v_n)|(\varphi_{j,\varepsilon} u_n,\varphi_{j,\varepsilon} v_n)\rangle
=&\int_{\Omega} \varphi_{j,\varepsilon} d\mu+\int_{\Omega} \varphi_{j,\varepsilon} d\tilde\mu
  -\int_{\Omega}\varphi_{j,\varepsilon} d\nu\\
-2\sqrt{\gamma}&\int_{\Omega}\!\! \varphi_{j,\varepsilon} u_0 v_0dx
  +\!\!\int_{\Omega}\!\! u_0 \left\langle \nabla u_0,\nabla \varphi_{j,\varepsilon}\right\rangle dx
    +\!\!\int_{\Omega}\!\! v_0 \left\langle \nabla v_0,\nabla \varphi_{j,\varepsilon}\right\rangle dx.
\end{align*}
By construction,
\begin{equation*}
\lim_{\varepsilon \to 0}\left[-2\sqrt{\gamma}\int_{\Omega} \varphi_{j,\varepsilon} u_0 v_0dx
  +\int_{\Omega} u_0 \left\langle \nabla u_0,\nabla \varphi_{j,\varepsilon}\right\rangle dx
    +\int_{\Omega} v_0 \left\langle \nabla v_0,\nabla \varphi_{j,\varepsilon}\right\rangle dx \right]=0.
\end{equation*}
Then, as $\mathcal{J}_{\gamma}'(u_n)\to0$ in $\left(H_0^1(\Omega)\times H_0^1(\Omega)\right)'$, we obtain that,
\begin{equation*}
\lim_{\varepsilon \to 0}  \left(\int_{\Omega}\varphi_{j,\varepsilon} d\mu
                                +\int_{\Omega}\varphi_{j,\varepsilon} d\tilde\mu
                                                                -\int_{\Omega}\varphi_{j,\varepsilon} d\nu\right)
                                            =\mu_j+\tilde\mu_j-\nu_j=0,
\end{equation*}
and we conclude
\begin{equation}\label{cons:concentration_system}
\nu_j= \mu_j+\tilde\mu_j.
\end{equation}
Finally, we have two options either the PS sequence has a convergent subsequence or it concentrates
around some of the points $x_j$. In other words, $\nu_j=\mu_j=\tilde\mu_j=0$, or there exists some $\nu_j>0$
such that, by \eqref{mescon} and \eqref{cons:concentration_system}, $\nu_j \geq S_N^{N/2}$. In case of having concentration,
we find that
\begin{align*}
  c=&\lim_{n \to \infty} \mathcal{J}_{\gamma}(u_n,v_n)
       =\lim_{n \to \infty} \mathcal{J}_{\gamma}(u_n,v_n)-\frac{1}{2}\langle\mathcal{J}_{\gamma}(u_n,v_n)| (u_n,v_n)\rangle\\
   =&\left(\frac{1}{2}-\frac{1}{2^*}\right)\int_{\Omega}|u_0|^{2^*}dx+\left(\frac{1}{2}-\frac{1}{2^*}\right)\nu_j\\
\geq&\frac{1}{N}S_N^{N/2}=c^*,
\end{align*}
in contradiction with the hypotheses $c<c^*$. Therefore, the PS sequence has a convergent
subsequence and the PS condition is satisfied.
\end{proof}
Next we show that we can obtain a path for $\mathcal{J}_{\gamma}$
under the critical level $c^*$. To obtain such path we will assume test functions
of the form
\begin{equation*}
(\tilde u_{\varepsilon},\tilde v_{\varepsilon})=(M\phi_{\varepsilon},M\rho \phi_{\varepsilon}),
\end{equation*}
where
\begin{equation*}
\phi_{\varepsilon}=\varphi_{j,R} \; u_{j,\varepsilon},
\end{equation*}
with $\varphi_{j,R}$ is a cut-off function defined by \eqref{cutoff2}, for some $R>0$ small enough,
$M>0$ a sufficiently large constant such that $\mathcal{J}_{\gamma} (\tilde{u}_{\varepsilon},\tilde{v}_{\varepsilon})<0$,
$\rho$ is a positive term to be determined below and $u_{j,\varepsilon}$ are the family of functions
defined by \eqref{extrem}. For the sake of simplicity, in the sequel we will consider $x_j=0$ as well as
the normalization \eqref{norms2}.\newline
Then, under the previous construction, we define the set of paths
\begin{equation*}
\Gamma_{\varepsilon}:=\left\{g\in C\left([0,1],H_0^1(\Omega)\times H_0^1(\Omega)\right)\,;\, g(0)=(0,0),\; g(1)=(\tilde u_{\varepsilon},\tilde v_{\varepsilon})\right\},
\end{equation*}
and consider the minimax value
\begin{equation*}
c_{\varepsilon}=\inf_{g\in\Gamma_{\varepsilon}} \max_{t \in [0,1]} \mathcal{J}_{\gamma}(g(t)).
\end{equation*}
Now we prove that, in fact, the levels $c_{\varepsilon}$ are always below $c^*$ for $\varepsilon>0$ small enough.

\begin{lemma}\label{level}
Assume $p=2^*-1$. Then, there exists $\varepsilon>0$ small enough such that,
\begin{equation*}
\sup_{0\leq t\leq1} \mathcal{J}_{\gamma}(t\tilde u_{\varepsilon},t\tilde v_{\varepsilon})< \frac{1}{N} S_N^{N/2},
\end{equation*}
provided $N\geq7$.
\end{lemma}
\begin{proof}
Let us denote by $F(\varepsilon)$ the estimate \eqref{esl2} in Lemma \ref{lees}. Then, assuming
the normalization \eqref{norms2},
\begin{align*}
g(t):=\mathcal{J}_{\gamma}(t\tilde u_{\varepsilon},t\tilde v_{\varepsilon})
&=\left(\frac{t^2M^2}{2}+\frac{\rho^2t^2M^2}{2} \right) \|\nabla \phi_{\varepsilon}\|_{L^2(\Omega)}^2
  -t^2 M^2\rho \sqrt{\gamma}\int_{\Omega} \phi_{\varepsilon}^2dx
  -\frac{t^{2^*}M^{2^*}}{2^*}\\
&=\frac{t^2M^2}{2}\left(\left(1+\rho^2\right) [S_N+ O(\varepsilon^{N-2})]
  -2\rho\sqrt{\gamma} F(\varepsilon)\right)-\frac{t^{2^*}M^{2^*}}{2^*}.
\end{align*}
It is clear that $\displaystyle \lim_{t\to \infty} g(t)=-\infty$, therefore, the function $g(t)$ possesses
a maximum value at the point,
\begin{equation*}
t_{\varepsilon}=\left(\frac{M^2\left[\left(1+\rho^2\right)[S_N+O(\varepsilon^{N-2})]
                -2\rho\sqrt{\gamma}F(\varepsilon)\right]}{M^{2^*}}\right)^{\frac{1}{2^*-2}}.
\end{equation*}
Moreover, at this point $t_{\varepsilon}$,
\begin{equation*}
g(t_{\varepsilon})=\frac{1}{N}\left[\left(1+\rho^2\right)[S_N+O(\varepsilon^{N-2})]
                   -2\rho\sqrt{\gamma}F(\varepsilon)\right]^{N/2}.
\end{equation*}
Then, the proof will be completed if we can choose $\rho>0$ such that the inequality,
\begin{equation}\label{csi}
\left[\left(1+\rho^2\right)[S_N+O(\varepsilon^{N-2})]-2\rho\sqrt{\gamma} F(\varepsilon)\right]<S_N,
\end{equation}
holds true provided $\varepsilon>0$ is small enough. Indeed, if we take $\rho={\varepsilon}^{\alpha}$,
with $\alpha>0$ (to be determined), inequality \eqref{csi} is equivalent to
\begin{equation*}
S_N\varepsilon^{2\alpha}+O(\varepsilon^{N-2+2\alpha})+O(\varepsilon^{N-2})
   <2\sqrt{\gamma}\varepsilon^{\alpha}F(\varepsilon),
\end{equation*}
Since $S_N\varepsilon^{2\alpha}+O(\varepsilon^{N-2+2\alpha})+O(\varepsilon^{N-2})=O(\varepsilon^{\tau})$
with $\tau=\min\{2\alpha, N-2+2\alpha, N-2\}=\min\{2\alpha, N-2\}$, we are left to prove that we
can choose $\alpha>0$ such that,
\begin{equation}\label{csi2}
O(\varepsilon^{\tau})<2\sqrt{\gamma}\varepsilon^{\alpha}
\cdot\left\{\begin{array}{ll}
C\varepsilon+O(\varepsilon^2), & \hbox{if}\quad N=3,\\
\frac{C\varepsilon^2}{2} |\log \varepsilon|+O(\varepsilon^2), & \hbox{if}\quad N=4,\\
C\varepsilon^2+O(\varepsilon^{N-2}), & \hbox{if}\quad N\geq 5.
\end{array}\right.
\end{equation}
provided $\varepsilon>0$ is small enough.
\begin{itemize}
\item If $N=3$, the corresponding inequality in \eqref{csi2} holds true if $\tau=\min\{2\alpha,1\}>\alpha+1$
that is not possible.
\item If $N=4$, the corresponding inequality \eqref{csi2} holds true if
\begin{equation*}
O(\varepsilon^{\tau})
<C\sqrt{\gamma}\varepsilon^{2\alpha+2}|\log\varepsilon|\quad\Rightarrow\quad O(\varepsilon^{\tau-2-\alpha})
<C\sqrt{\gamma}|\log\varepsilon|,
\end{equation*}
and thus, necessarily $\tau=\min\{2\alpha,2\}>2+\alpha$, that, once again, is not possible.
\item If $N\geq5$, the corresponding inequality \eqref{csi2} holds true if $\tau=\min\{2\alpha,N-2\}>2+\alpha$.
Let us observe that $\displaystyle \min\{a,b\}=\frac{1}{2}\left(a+b-|a-b|\right)$, hence,
inequality \eqref{csi2} will be satisfied if we can choose $\alpha>0$ such that
\begin{equation}\label{lst}
N-|2\alpha-(N-2)|>6.
\end{equation}
Now we have two options, either $2\alpha>N-2$ or $2\alpha<N-2$.
\begin{itemize}
\item In the first case, thanks to inequality \eqref{lst}, we find the condition
$\frac{N}{2}+1>N-\alpha>4$, that can be fulfilled only for $N>6$.
\item In the second case, thanks to inequality \eqref{lst}, we find the condition
$N-2>2\alpha>4$, that can be fulfilled, once again, only for $N>6$.
\end{itemize}
\end{itemize}
Thus, if $N\geq 7$ we can choose $\alpha>2$ such that \eqref{csi2} is satisfied. Finally,
note that with the assumption $\rho={\varepsilon}^{\alpha}$ we have
\begin{equation*}
t_{\varepsilon}=\left(\frac{M^2\left[\left(1+\rho^2\right)[S_N+O(\varepsilon^{N-2})]
                      -2\rho\sqrt{\gamma}F(\varepsilon)\right]}{M^{2^*}}\right)^{\frac{1}{2^*-2}}
                      \geq\delta>0,
\end{equation*}
provided $\varepsilon>0$ is small enough.
\end{proof}

\begin{proof}{Proof of Theorem \ref{Thsystem}. Critical case.}\hfill\break
Thanks to Lemma \ref{lezero_sys} and Lemma \ref{level},we find that
\begin{equation*}
0<c_{\varepsilon}\leq \sup_{0\leq t\leq1} \mathcal{J}_{\gamma}(t\tilde u_{\varepsilon},t\tilde v_{\varepsilon})
               < \frac{1}{N} S_N^{N/2},
\end{equation*}
provided $\varepsilon>0$ is small enough. Because of Lemma \ref{lezero_sys} the functional
$\mathcal{J}_{\gamma}$ has the MPT geometry. Moreover, because of Lemma \ref{PScondition} the
functional $\mathcal{J}_{\gamma}$ satisfies the PS condition for any level $c_{\varepsilon}$ with
$\varepsilon>0$ small enough. Therefore, we can apply the Mountain Pass Theorem and conclude
the existence of a critical point $(u,v)\in H_0^1(\Omega)\times H_0^1(\Omega)$. The rest
follows as in the subcritical case.
\end{proof}

\section{Further Extensions}\label{furext}
Let us consider the following high-order problem with generalized Navier boundary conditions,
\begin{equation}\label{extension}
\left\{
\begin{array}{rllll}
(-\Delta)^{m+1} u\!\!\!\! &=\gamma u+ (-\Delta)^{m}|u|^{p-1}u& & &\mbox{in}\quad \Omega\subset \mathbb{R}^{N}, \\
(-\Delta)^ju \!\!\!\!     &=0 &\mbox{ for}\ 0\leq j\leq m,     & &\mbox{on}\quad \partial\Omega,
\end{array}
\right.
\tag{$P_{\gamma}^{m+1}$}
\end{equation}
with $m$ a natural number bigger than 1, and the variational problem obtained
applying the operator $(-\Delta)^{-m}$ to \eqref{extension},
\begin{equation}\label{extension2}
\left\{
\begin{tabular}{lcl}
$-\Delta u=\gamma (-\Delta)^{-m}u+ |u|^{p-1}u$ & &in $\Omega\subset \mathbb{R}^{N}$, \\
\quad\ \ $u=0$                                        & &on $\partial\Omega$.
\end{tabular}
\right.
\tag{$E_{\gamma,m}$}
\end{equation}
associated with the following Euler-Lagrange functional,
\begin{equation*}
\mathcal{F}_{\gamma,m}(u)=\frac{1}{2} \int_{\Omega} |\nabla u|^2dx
                     -\frac{\gamma}{2}\int_{\Omega} |(-\Delta)^{-m/2} u|^2 dx
                                         -\frac{1}{p+1} \int_{\Omega}|u|^{p+1}dx.
\end{equation*}
Note that, as it happens for $m=1$, the embedding features for problem
\eqref{extension2} are governed by the standard second-order equation,
\begin{equation*}
-\Delta u=|u|^{p-1}u,
\end{equation*}
thus, the variational framework coincides with the one of the case
$m=1$, so that we also consider $1<p\leq 2^*-1$.\newline Let us
observe that if we try to prove the existence of a positive solution
to problem \eqref{extension2} directly as performed for the problem
\eqref{eq:nonlocal} in Section \eqref{concomp}, we immediately run
into complications.

Due to the lack of a comparison principle, we can not use a similar
argument to Lemma \eqref{lem:estimacion_nolocal} when dealing with
the operator $(-\Delta)^{-m}$. Thus, we will make full use of the
correspondence between problem \eqref{extension2} and the following
elliptic system,
\begin{equation}\label{extension_system}
\left\{
\begin{tabular}{l}
$-\Delta u=\gamma^{\frac{1}{m+1}}v_1+|u|^{p-1}u,$ \\
$-\Delta v_1=\gamma^{\frac{1}{m+1}}v_2,$\\
$-\Delta v_2=\gamma^{\frac{1}{m+1}}v_3,$\\
$\,\,\,\,\,\, \vdots$\\
$-\Delta v_{m}=\gamma^{\frac{1}{m+1}}u$\\
\end{tabular}
\right.
\tag{$S_{\gamma,m}$}
\quad \hbox{in}\quad \Omega,\quad (u,v_1,\ldots,v_m)=(0,0,\ldots,0)\quad\mbox{in } \partial\Omega,
\end{equation}
whose associated Euler-Lagrange functional is defined by
\begin{align}\label{functional_extension}
\mathcal{J}_{\gamma,m}\left(\mathcal{U}\right)
&=\frac{1}{2}\int_{\Omega}|\nabla u|^2dx+\frac{1}{2}\sum_{i=1}^m\int_{\Omega}|\nabla v_i|^2dx\nonumber\\
& -\frac{\gamma^{\frac{1}{m+1}}}{m+1}\left(\int_{\Omega}uv_1dx+\int_{\Omega}uv_mdx+\sum_{i=1}^{m-1}\int_{\Omega}v_iv_{i+1}dx\right)
  -\frac{1}{p+1}\int_{\Omega}|u|^{p+1}dx,
\end{align}
where $\mathcal{U}=(u,v_1,\ldots,v_m)$. The functional $\mathcal{J}_{\gamma,m}$  has the same structure as
the functional $\mathcal{J}_{\gamma}$ thus, the ideas developed in
Section \ref{concomp2} will fit, with slight variations, in this scenario.\newline
Let us denote by $\Lambda_1^*$ the first eigenvalue of the operator $(-\Delta)^{m+1}$
under the homogeneous generalized Navier boundary conditions given by \eqref{extension}.
It is clear from the spectral definition of the operator $(-\Delta)^{m+1}$ that
$\Lambda_1^*=\lambda_1^{m+1}$ with $\lambda_1$ the first eigenvalue of the Laplace operator
under homogeneous Dirichlet boundary conditions.\newline
The aim of this last section is then to prove the following.
\begin{theorem}\label{Thsystem_m_subcritical}
Assume $1<p<2^*-1$. Then, for every $\gamma\in (0,\Lambda_1^*)$,
there exists a positive solution to system \eqref{extension_system}.
\end{theorem}
\begin{theorem}\label{Thsystem_m}
Assume $p=2^*-1$. Then, for every $\gamma\in (0,\Lambda_1^*)$,
there exists a positive solution to system \eqref{extension_system} provided $N\geq 7$.
\end{theorem}

We start determining the interval of values of the parameter $\gamma>0$ compatible with
existence of positive solutions related to problem \eqref{extension2}.

\begin{lemma}\label{cotaeigenvalue}
Equation \eqref{extension2} does not possess a positive solution when
\begin{equation}
\gamma\geq\Lambda_1^*.
\end{equation}
\end{lemma}
\begin{proof}
Using as a test function in \eqref{extension2} the first eigenfunction $\varphi_1$ associated
with the first eigenvalue $\lambda_1$ for the Laplacian operator $(-\Delta)$ with homogeneous
Dirichlet boundary conditions together with $\Lambda_1^*=\lambda_1^{m+1}$ the result follows.
\end{proof}
Next we deal with the MPT conditions. We state the analogous results to those of the case $m=1$.
Since the proofs of the next results rely on the ideas developed for the case $m=1$, we will only
remark the main differences, if any.
\begin{lemma}
\label{lezero_sys_m}
The functional $\mathcal{J}_{\gamma,m}\left(\mathcal{U}\right)$ has the MPT geometry.
\end{lemma}
\begin{proof}
The proof is similar to the proof of Lemma \ref{lezero_sys} so we omit the details.
\end{proof}
\begin{lemma}\label{acotacion_sys_m}
Let $\mathbb{E}_m:=H_0^1(\Omega)\times H_0^1(\Omega)\times\ldots\times H_0^1(\Omega)$ and
$\left\{\mathcal{U}_n\right\}=\left\{(u_n,v_{1,n},\ldots,v_{m,n})\right\}\subset\mathbb{E}_m$ be a PS sequence
for the functional $\mathcal{J}_{\gamma,m}$, i.e.
\begin{equation*}
\mathcal{J}_{\gamma,m}(\mathcal{U}_n) \rightarrow c,\quad    \mathcal{J}_{\gamma,m}'(\mathcal{U}_n) \rightarrow 0,\quad \hbox{as}\quad n\to \infty.
\end{equation*}
Then,
\begin{equation*}
\left\{\mathcal{U}_n\right\}\quad \hbox{is bounded in}\quad \mathbb{E}_m.
\end{equation*}

\end{lemma}
\begin{proof}
Arguing as in the proof of Lemma \ref{acotacion_sys} we find,
\begin{align*}
&(m+1)\left(\frac{1}{2}-\mu\right)\left(1-\frac{2\gamma^{\frac{1}{m+1}}}{(m+1)\lambda_1}\right)
\left(\|u_n\|_{H_0^1(\Omega)}^2+\sum_{i=1}^{m}\|v_{i,n}\|_{H_0^1(\Omega)}^2\right)\\
&\leq (m+1)c+\left(\|u_n\|_{H_0^1(\Omega)}+\sum_{i=1}^{m}\|v_{i,n}\|_{H_0^1(\Omega)}\right)\cdot o(1).
\end{align*}
Keeping in mind Lemma \ref{cotaeigenvalue}, it follows that
\begin{equation*}
\left(\frac{1}{2}-\mu\right)\left(1-\frac{2\gamma^{\frac{1}{m+1}}}{(m+1)\lambda_1}\right)>0,
\end{equation*}
and we conclude the boundedness of the sequence $\{\mathcal{U}_n\}$ in $\mathbb{E}_m$.
\end{proof}
\begin{proof}[Proof of Theorem \ref{Thsystem_m_subcritical}.]\hfill\break
Combining Lemma \ref{lezero_sys_m} and Lemma \ref{acotacion_sys_m} together with the
Rellich-Kondrachov Theorem the hypotheses of the Mountain Pass Theorem are fulfilled and we
conclude as in the proof of Theorem \ref{Thsystem_subcritical}.
\end{proof}
To finish, we deal with the critical case $p=2^*-1$. As it was done in previous sections, with
the aid of a concentration-compactness argument we will prove that the PS condition is satisfied
for any level below the critical level
\begin{equation*}
c^*=\frac{1}{N}S_N^{N/2}.
\end{equation*}
Let us observe that the critical level $c^*$ is independent of the order of the inverse operator
involved in problem \eqref{extension2} as it coincides with the critical level for problem
\eqref{eq:nonlocal}.
\begin{lemma}\label{PScondition_sys}
The functional $\mathcal{J}_{\gamma,m}$ defined by \eqref{functional_extension} satisfies the Palais-Smale
condition for any level $c$ below the critical level $c^*$.
\end{lemma}
\begin{proof}
Let $\left\{\mathcal{U}_n\right\}=\left\{(u_n,v_{1,n},\ldots,v_{m,n})\right\}\subset\mathbb{E}_m$ be
a PS sequence of level $c<c^*$. Because of Lemma \ref{acotacion_sys_m} and
Lemma \ref{CC}, we can replicate the steps of the proof of Lemma \ref{PScondition} incorporating
the slight difference that, instead \eqref{cons:concentration_system}, we find now
\begin{equation}\label{cons:concentracion_sys}
\nu_j = \mu_j+\sum_{i=1}^{m}\tilde\mu_{i,j}.
\end{equation}
with
\begin{equation}\label{mesconsys}
\mu_j\geq S_N \nu_j^{2/2^*}.
\end{equation}
Then, either the PS sequence has a convergent subsequence or it concentrates around
some of the points $x_j$. In other words, $\nu_j=\mu_j=\tilde\mu_{i,j}=0$, or there exists some
$\nu_j>0$ such that, thanks to \eqref{cons:concentracion_sys} and \eqref{mesconsys},
$\nu_j \geq S_N^{N/2}$. In case of having concentration,
\begin{align*}
c&=\lim_{n \to \infty} \mathcal{J}_{\gamma,m}(\mathcal{U}_n)
  =\lim_{n \to \infty} \mathcal{J}_{\gamma,m}(\mathcal{U}_n)
                         -\frac{1}{2}\langle\mathcal{J}_{\gamma,m}(\mathcal{U}_n)| \mathcal{U}_n\rangle\\
 &=\left(\frac{1}{2}-\frac{1}{2^*}\right)\int_{\Omega}|u_0|^{2^*}dx+\left(\frac{1}{2}-\frac{1}{2^*}\right)\nu_j\\
 &\geq \frac{1}{N}S_N^{N/2}=c^*,
\end{align*}
in contradiction with the hypotheses $c<c^*$.
\end{proof}
Finally, we show that we can obtain a path for the
functional $\mathcal{J}_{\gamma,m}$ under the critical level $c^*$. Following the ideas of the previous
sections, we will assume test functions of the form
\begin{equation}\label{test3}
\tilde{\mathcal{U}}_{\varepsilon}=(\tilde{u}_{\varepsilon},\tilde{v}_{1,\varepsilon},\ldots,\tilde{v}_{m,\varepsilon})
                                 =(M\phi_{\varepsilon},M\rho\phi_{\varepsilon},\ldots,M\rho\phi_{\varepsilon}),
\end{equation}
with $M>0$ a sufficiently large constant so that $\mathcal{J}_{\gamma,m}(\tilde{\mathcal{U}}_{\varepsilon})<0$,
$\rho$ is positive term to be determined as in the case $m=1$, and $u_{j,\varepsilon}$ are the family of functions
defined by \eqref{extrem}. As performed above we will consider $x_j=0$. Then, under the previous
construction, let us define the set of paths
\begin{equation*}
\Gamma_{\varepsilon}:=\{g\in C([0,1],\mathbb{E}_m)\,;\, g(0)=\overline{0},\; g(1)=\tilde{\mathcal{U}}_{\varepsilon}\},
\end{equation*}
and consider the minimax value
\begin{equation*}
c_{\varepsilon}=\inf_{g\in\Gamma_{\varepsilon}} \max_{t \in [0,1]} \mathcal{J}_{\gamma,m}(g(t)).
\end{equation*}
Next, we  check that any level $c_{\varepsilon}$ is always below $c^*$ provided $\varepsilon>0$ is small enough.
This is done thanks to Lemma \ref{lees}.

\begin{lemma}\label{level_sys}
Assume $p=2^*-1$ and $N\geq7$. Then, there exists $\varepsilon>0$ small enough such that,
\begin{equation*}
\sup_{0\leq t\leq 1} \mathcal{J}_{\gamma,m}(t\tilde{\mathcal{U}}_{\varepsilon})< \frac{1}{N} S_N^{N/2}.
\end{equation*}
\end{lemma}
\begin{proof}
Let us denote by $F(\varepsilon)$ the estimate \eqref{esl2} in Lemma \ref{lees}. Then, assuming the normalization \eqref{norms2}, we obtain
\begin{align*}
g(t):=&\mathcal{J}_{\gamma,m}(t\tilde{\mathcal{U}}_{\varepsilon})\\
     =&\left(\frac{1}{2}(1+m\rho^2)[S_N+O(\varepsilon^{N-2})]
           -\frac{\gamma^{\frac{1}{m+1}}}{m+1}(2\rho+(m-1)\rho^2)F(\varepsilon)\right)M^2t^2
           -\frac{M^{2^*}t^{2^*}}{2^*}.
\end{align*}
Proceeding a in the proof of Lemma \ref{level}, the proof will be completed if we can choose
$\rho>0$ such that the inequality,
\begin{equation*}
O(\varepsilon^{N-2})+m\rho^2S_N+m\rho^2O(\varepsilon^{N-2})
<2\frac{\gamma^{\frac{1}{m+1}}}{m+1}(2\rho+(m-1)\rho^2)F(\varepsilon),
\end{equation*}
holds true provided $\varepsilon>0$ is small enough. We take $\rho=\varepsilon^{\alpha}$ with $\alpha>0$
(to be determined) and $\displaystyle\tau=\min\{N-2,2\alpha,2\alpha+N-2\}=\min\{N-2,2\alpha\}$. Then,
since $O(\varepsilon^{\alpha}+\varepsilon^{2\alpha})=O(\varepsilon^{\alpha})$, we are left to prove that
for a constant $C>0$ the inequality,
\begin{equation}\label{lst2}
O(\varepsilon^{\tau})<C\varepsilon^{\alpha}F(\varepsilon),
\end{equation}
holds true provided $\varepsilon>0$ is small enough. Since inequality \eqref{lst2} coincides with \eqref{csi2} the
arguments performed in Lemma \ref{level} allow us to conclude.
\end{proof}

\begin{proof}{Proof of Theorem \ref{Thsystem_m}.}\hfill\break
Thanks to Lemma \ref{lezero_sys} and Lemma \ref{level},we find that
\begin{equation*}
c_{\varepsilon}\leq \sup_{t\geq 0} \mathcal{J}_{\gamma}(t\tilde{\mathcal{U}}_{\varepsilon})< \frac{1}{N} S_N^{N/2},
\end{equation*}
provided $\varepsilon>0$ is sufficiently small. Hence, combining Lemma \ref{lezero_sys_m} and
Lemma \ref{PScondition_sys} we can apply the Mountain Pass Theorem and conclude the existence of a critical
point $\mathcal{U}\in\mathbb{E}_m$. The rest follows as in the former cases.
\end{proof}



\begin{thebibliography}{99}

\bibitem{PV} P.~\'{A}lvarez-Caudevilla and V.A.~Galaktionov, \textit{Steady states, global existence and blow-up for fourth-order
semilinear parabolic equations of Cahn--Hilliard type}. Adv. Nonlinear Stud. \textbf{12} (2012), no. 315--361.

\bibitem{AEGnegII} P.~\'Alvarez-Caudevilla, J.D.~Evans and V.A.~Galaktionov, \textit{Countable families of solutions of a limit stationary
semilinear fourth-order Cahn--Hilliard-type equation I. Mountain pass and Lusternik--Schnirel'man patterns in $\mathbb{R}^N$}. Bound. Value Prob. 2016, Paper No. 171, 25pp.

\bibitem{AR} A. Ambrosetti and P.H. Rabinowitz, \textit{Dual variational methods in critical point theory and applications}. J. Funct. Anal. \textbf{14} (1973), 349--381.

\bibitem{ABC} A. Ambrosetti, H. Brezis and G. Cerami, \textit{Combined effects of concave and convex nonlinearities in some elliptic problems}.
J. Funct. Anal. \textbf{122} (1994), no. 2, 519--543.

\bibitem{BN} H. Brezis and L. Nirenberg, \textit{Positive solutions of nonlinear elliptic equations involving critical Sobolev exponents}.
Comm. Pure Appl. Math. \textbf{36} (1983), no. 4, 437--477.

\bibitem{EV} Lawrence C. Evans, \textit{Partial Differential Equations}. AMS, Graduate Studies in Mathematics \textbf{19}, (1998).

\bibitem{GGS} F. Gazzola, C-H. Grunau and G. Sweers, \textit{Polyharmonic boundary value problems. Positivity preserving and nonlinear higher order elliptic equations in bounded domains}. Lecture Notes in Mathematics, 1991. Springer--Verlag, Berlin, 2010. xviii+423 pp.

\bibitem{Lions} P.-L. Lions, \textit{The concentration-compactness principle in the calculus of variations. The limit case. II}. Rev. Mat. Iberoamericana, \textbf{1} (1985), no.2, 45--121.

\bibitem{Poh}  S. Pohozahev, \textit{On the eigenfunctions of the equation $\Delta u+ \lambda f(u)=0$}. Dokl. Akad. Nauk SSSR {\bf
165} (1965) 36-39.

\bibitem{Ta}  G. Talenti, \textit{Best constant in Sobolev inequality}.
Ann. Mat. Pura Appl. (4) \textbf{110} (1976), 353--372.

\end{thebibliography}
\end{document}